\documentclass[]{article}
\usepackage[affil-it]{authblk}

\usepackage{fullpage}
\usepackage{graphicx}
\usepackage{pdflscape}
\usepackage{secdot}
\usepackage{subcaption}
\usepackage{amsmath, amssymb,amsthm}
\usepackage{amsfonts}
\usepackage{latexsym}
\usepackage{booktabs}
\usepackage{enumitem,blindtext}
\usepackage[numbers]{natbib} %
\usepackage[export]{adjustbox}
\usepackage{mathrsfs}
\usepackage[usenames,dvipsnames]{color}
\usepackage[mathcal]{eucal}
\usepackage{caption}
\usepackage{amsmath,amsfonts}
\usepackage[linesnumbered, noend]{algorithm2e}
\usepackage{setspace}
\newcommand{\comments}[1]{}
\DeclareGraphicsExtensions{.png,.pdf,.jpg}
\newtheorem{theorem}{Theorem}
\newtheorem{definition}{Definition}

\usepackage{tikz}
\usepackage{tikzsymbols}
\usetikzlibrary{arrows}
\usetikzlibrary{calc}

\newcommand{\WTDP}{WTDP}
\newcommand{\TDP}{TDP}
\newcommand{\UFL}{UFL}

\newcommand{\maone}{(MA1)}
\newcommand{\matwo}{(MA2)}
\newcommand{\mathree}{(MA3)}
\newcommand{\ffirst}{(F1)}
\newcommand{\fsecond}{(F2)}

\newcommand{\MAINST}{\texttt{MA}}
\newcommand{\NEWINST}{\texttt{NEW}}

\renewcommand{\paragraph}[1]{{\vskip2mm\noindent \bf #1}\hspace{0.4cm}}

\definecolor{purple}{RGB}{153,50,204}

\usepackage[normalem]{ulem}
\DeclareMathOperator*{\argmin}{arg\,min}

\usepackage{hyperref}
\begin{document}



\title{Exact and heuristic algorithms for the weighted total domination problem}

\author[1,2]{Eduardo \'Alvarez-Miranda\thanks{ealvarez@utalca.cl}}

\author[3]{Markus Sinnl\thanks{markus.sinnl@jku.at}}

\affil[1]{Department of Industrial Engineering, Faculty of Engineering, Universidad de Talca, Campus Curic\'o, Chile}
\affil[2]{Instituto Sistemas Complejos de Ingenier\'ia ISCI, Chile}

\affil[3]{Institute of Production and Logistics Management, Johannes Kepler University Linz, Linz, Austria}

\date{}
\maketitle

\begin{abstract}
Dominating set problems are among the most important class of combinatorial problems in graph optimization,
from a theoretical as well as from a practical point of view. In this paper, we address the recently introduced (minimum) weighted total domination problem. 
In this problem, we are given an undirected graph with a vertex weight function and an edge weight function. 
The goal is to find a total dominating set $D$ in this graph with minimal weight. 
A total dominating set $D$ is a subset of the vertices such that every vertex in the graph, including vertices in $D$, is adjacent to a vertex in $D$. 
The weight is measured as the sum of all vertex weights of vertices in $D$, plus the edge weights in the subgraph induced by $D$, plus for each vertex not in $D$ the minimum weight of an edge from it to a vertex in $D$. 

In this paper, we present two new Mixed-Integer Programming models for the problem, and design solution frameworks based on them.
These solution frameworks also include valid inequalities, 
starting heuristics and primal heuristics. 
In addition, we also develop a genetic algorithm, which is based on a greedy randomized adaptive search procedure version of our starting heuristic. 

We carry out a computational study to assess the performance of our approaches when compared to the previous work
for the same problem. 
The study reveals that our exact solution algorithms are up to 500 times faster compared to previous exact approaches and instances with up to 125 vertices can be solved to optimality within a timelimit of 1800 seconds. 
Moreover, the presented genetic algorithm also works well and often finds the optimal or a near-optimal solution within a short runtime. Additionally, we also analyze the influence of instance-characteristics on the performance of our algorithms.
\end{abstract}

\section{Introduction and motivation}

\emph{Dominating set problems} are among the most important class of combinatorial problems in graph optimization,
from a theoretical as well as from a practical point of view. 
For a given graph $G = G(V,E)$, a subset $D\subset V$ of vertices is referred to as a \textit{dominating set} if
the remaining vertices, i.e., $V\setminus D$, are \textit{dominated} by $D$ according to a given topological relation (e.g., they are all adjacent to at least one vertex from $D$).
Dominating set problems (also often called \emph{domination problems} in graphs)  have attracted the attention of computer scientists and applied mathematicians
since the early 50s, and their close relation to covering and independent set problems has lead to the development of a whole research area (see, e.g.,~\cite{Ore1962} and~\cite{doi:10.1002/net.3230070305} for early references on domination problems).

There are many applications where set domination and related concepts play a central role, including school bus routing~\cite{7502983}, communication networks~\cite{wan2002distributed}, radio station location~\cite{erwin2004dominating}, social networks analysis~\cite{SunMa2017}, biological networks analysis~\cite{NACHER201657}, and also chess-problems like the five-queens problem~\cite{rolfes2014copositive}; 
see e.g., the book~\cite{haynes2013fundamentals} for a comprehensive overview of domination problems. 
Variants of dominating set problems include e.g., the \emph{connected dominating set problem}~\cite{DuWan2013}, the \emph{(weighted) independent dominating set problem}~\cite{GODDARD2013839,PINACHODAVIDSON2018860}, among others (see,~e.g.,~\cite{Kang2013} for further variants of the dominating set problems).



In this paper, we address the recently introduced \emph{(minimum) weighted total domination problem (\WTDP)} which is defined as follows. 

\begin{definition}Let $\mathbf{w}: V\rightarrow \mathbb{R}_{\geq 0}^{|V|}$ be a vertex weight function,
and let $\mathbf{c}: E\rightarrow \mathbb{R}_{\geq 0}^{|E|}$  be an edge weight function.
The weighted total domination problem is the problem of finding a set $D\subset V$, 
such that every vertex in $V$ (including the vertices from $D$) has at least one neighbor in $D$
and the function
\begin{align}
w(D) = \sum_{i\in D} w_i + \sum_{e\in E(D)} c_e + \sum_{i \in V\setminus D} \min \{c_e\mid e:\{i,j\}\in E\;\mbox{and}\;j\in N(i)\cap D\} \notag
\end{align}is minimized, where $E(D)\subseteq E$ corresponds to the set of edges \emph{inside} $D$,
and $N(i)\subset V$ corresponds to the set of neighboring vertices of vertex $i\in V$.
\end{definition}

For referring to the different components of the objective function, we denote $\sum_{i\in D} w_i$ as the \emph{vertex selection costs}, $\sum_{e\in E(D)} c_e$ as the \emph{internal edge costs} and $\sum_{i \in V\setminus D} \min \{c_e\mid e:\{i,j\}\in E\;\mbox{and}\;j\in N(i)\cap D\}$ as the \emph{external edge costs}.
Figure~\ref{fig:example} gives an exemplary instance of the \WTDP, together with its optimal solution.
\tikzstyle{vertex}=[circle,fill=black!15,minimum size=20pt,inner sep=0pt]
\tikzstyle{edge} = [draw,thick,-]
\tikzstyle{weight} = [draw=none,fill=white,inner sep=1pt, font=\small]
\tikzstyle{selected edge} = [draw,line width=2pt,-,black!100]
\tikzstyle{unselected vertex}=[circle,fill=black!5,minimum size=20pt,inner sep=0pt]
\begin{figure}[h!tb]
	\begin{subfigure}[b]{.5\linewidth}
		\centering
		\begin{tikzpicture}[scale=1.9]
		\foreach \pos/\name/\type/\weight in {{(0,2)/A/vertex/1}, {(1,2)/B/vertex/8}, {(2,2)/C/vertex/1},{(0,1)/D/vertex/5},{(1,1)/E/vertex/1},{(2,1)/F/vertex/7},{(0,0)/G/vertex/1},{(1,0)/H/vertex/9},{(2,0)/I/vertex/1}}
		\node[\type] (\name) at \pos {$\weight$};
		\foreach \source/ \dest/\weight in  
		{A/B/6,B/C/7,A/D/2,D/E/5,B/E/3,C/F/3,E/F/3,D/G/3,G/H/2,E/H/6,H/I/2,F/I/4}
		\path[edge] (\source) -- node[weight] {$\weight$} (\dest);

		\end{tikzpicture}
		\caption{Instance}\label{fig:instance}
	\end{subfigure}%
	\begin{subfigure}[b]{.5\linewidth}
		\centering
		\begin{tikzpicture}[scale=1.9]
	\foreach \pos/\name/\type/\weight in {{(0,2)/A/unselected vertex/}, {(1,2)/B/unselected vertex/}, {(2,2)/C/vertex/1},{(0,1)/D/vertex/5},{(1,1)/E/unselected vertex/},{(2,1)/F/vertex/7},{(0,0)/G/vertex/1},{(1,0)/H/unselected vertex/},{(2,0)/I/unselected vertex/}}
\node[\type] (\name) at \pos {$\weight$};

\foreach \source/ \dest/\weight in  
{A/D/2,B/C/7,C/F/3,E/F/3,D/G/3,F/I/4,G/H/2}
\path[selected edge] (\source) -- node[weight] {$\weight$} (\dest);

		\end{tikzpicture}

		\caption{Optimal solution}\label{fig:solution}
	\end{subfigure}
	\caption{Instance $I=(G=(V,E,\mathbf{c},\mathbf{w})$ and optimal solution with weight $14+6+18=38$ (\emph{vertex selection costs}+\emph{internal edge costs}+\emph{external edge costs}). We note that a solution consisting of all the vertices with weight one would not be feasible, as it is not a total dominating set, but only a dominating set. \label{fig:example}}
\end{figure}
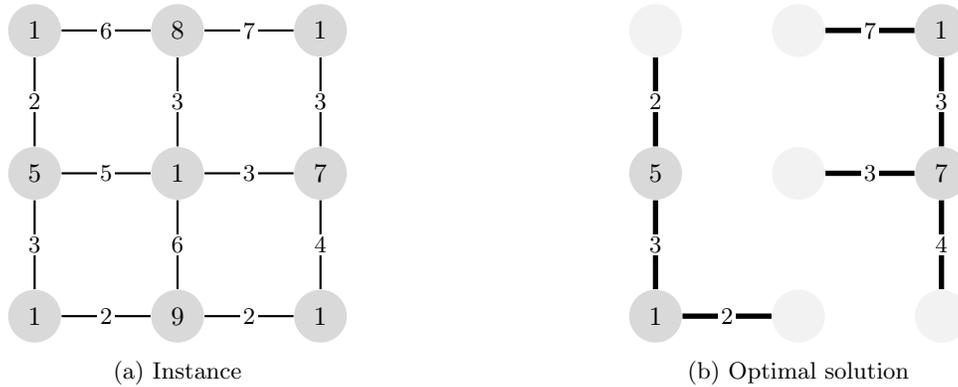

We note that in the \WTDP, we are not just concerned with the concept of domination, but with the stronger concept of \emph{total domination}, which imposes that for each vertex $v \in D$, there is also a neighbor of $v$ in $D$ (i.e., the vertices of $D$ also need to be dominated by $D$). 
The \WTDP\ was introduced in~\cite{MaEtAl2019} an is an extension of the (unweighted) total domination problem (\TDP), resp., the vertex-weighted total domination problem. 
In the \TDP, the objective function has $w_i=1$ for all $i \in V$, and $c_e=0$ for all $e \in E$. 
The optimal solution of the \TDP\ for a given graph is called its \emph{total domination number}. 
The \TDP\ was introduced in the 1980s (see~\cite{cockayne1980total}) 
and is NP-hard in general graphs (see~\cite{laskar1984algorithmic}, for further details). 
The \TDP\ has a rich history of research focusing on theoretical results, e.g., computational complexity and bounds for the domination number for certain graph classes, we refer the reader to the survey~\cite{henning2009survey} and the book~\cite{henning2013total} for more details. 
Applications of total domination include the design of communication networks and the forming of committees~\cite{haynes2013fundamentals, henning2004restricted}.

\paragraph{Contribution and Outline}
The \WTDP\ was recently introduced in~\cite{MaEtAl2019}, where three Mixed-Integer Programming (MIP) formulations to solve the problem were presented and evaluated in a computational study.
In this paper, we present two new MIP models for the problem, and design solution frameworks based on them.
These solution frameworks also include valid inequalities, 
starting heuristics and primal heuristics. A genetic algorithm (GA) is also developed,
which is based on a greedy randomized adaptive search procedure (GRASP) version of our starting heuristic.
We carry out a computational study to assess the performance of our new approaches in comparison to the previous work by~\cite{MaEtAl2019}. 
The study reveals that our algorithms are up to 500 times faster and instances with up to 125 vertices can be solved to optimality within a timelimit of 1800 seconds. Moreover, the presented heuristics (i.e., the GA, and just using the GRASP on its own) also work well and often find the optimal, or a near-optimal solution within a short runtime. 
Furthermore, we also analyze the influence of instance-characteristics on the performance of our algorithms.

The paper is organized as follows. In the reminder of this section, we give a short overview of the models introduced in~\cite{MaEtAl2019}. 
In Section~\ref{sec:form} we present our two new MIP models, together with valid inequalities. 
In Section~\ref{sec:bc} we discuss implementation details of the branch-and-cut algorithms we designed based on our new models, including a description of the starting and primal heuristics. 
In Section~\ref{sec:genalg}, we describe our genetic algorithm.
Section~\ref{sec:compres} contains our computational study, and concluding remarks are provided in Section~\ref{sec:con}.

\subsection{Revisiting the models of~\cite{MaEtAl2019}}
\label{subsec:prevwork}

In the following, we give a brief overview of the three formulations for the \WTDP\ presented in~\cite{MaEtAl2019} (we denote them as \maone, \matwo, \mathree). We re-implemented these models and included them in our computational study, see Section~\ref{sec:compres}.   

Firstly, consider the following set of variables and constraints which are common to all formulations of~\cite{MaEtAl2019} and that will also be part of our formulations. 
Let $\mathbf{x}\in\{0,1\}^{|V|}$ be a vector of binary variables, such that $x_i = 1$ if vertex $i\in V$ is taking as part of the (total) dominating set, and $x_i =0$ otherwise.
Constraints
\begin{equation}
\sum_{j\in N(i)} x_j \geq 1,\;\forall i\in V,\tag{TDOM}\label{eq:dom}
\end{equation}
ensure that the variables with $x_i=1$ form a total dominating set. We observe that these constraints are already enough to define the set of feasible solutions. The remaining constraints in the presented models are used to correctly measure the objective function.
Let $\mathbf{y}\in\{0,1\}^{|E|}$ be a vector of binary variables
associated with the edges $E$. 
These variables will be used in all formulations except of \mathree, and they are used differently depending on the considered formulation.
In \maone, \matwo, they are used to measure the contribution of any edge $e=\{i,j\}$ on the objective function, for both the internal edge costs and the external edge costs. In contrast, in the new formulations presented in Section~\ref{sec:form}, these variables are only used for the internal edge costs, and the external edge costs are modeled in different ways. 
\paragraph{Formulation \maone} 
Let $\mathbf{z}\in\{0,1\}^{|E|}$ be a vector of binary variables, 
such that $z_{e=\{i,j\}} = \min\{x_i,x_j\}$ for every edge $e:\{i,j\}\in E$.
For a given vertex $i\in V$, let $\delta(i)\subset E$ be the set of edges incident to $i$.
Formulation \maone~ is given by:
\begin{align}
\mbox{\maone}\quad\quad\quad w^* = \min &\sum_{i\in V} w_i x_i+ \sum_{e\in E} c_e y_e  \tag{WTD1.1}\label{eq:wtd11}\\
\mbox{s.t.}\quad\quad & \eqref{eq:dom} \notag \\
&x_i + x_j \geq y_e,\;\forall e:\{i,j\}\in E\tag{WTD1.3}\label{eq:wtd13}\\
&x_i \geq z_e\;\mbox{and}\; x_j\geq z_e,\;\forall e:\{i,j\}\in E\tag{WTD1.4}\label{eq:wtd14}\\
&z_e \geq x_i + x_j - 1,\;\forall e:\{i,j\}\in E\tag{WTD1.5}\label{eq:wtd15}\\
&y_e \geq z_e,\;\forall e\in E\tag{WTD1.6}\label{eq:wtd16}\\
&x_i + \sum_{e\in\delta(i)} y_e \geq 1,\;\forall i\in V\tag{WTD1.7}\label{eq:wtd17}\\
&\mathbf{x}\in\{0,1\}^{|V|},\;\mathbf{y}\in\{0,1\}^{|E|}\;\mbox{and}\;\mathbf{z}\in\{0,1\}^{E} \notag
\end{align}

\paragraph{Formulation \matwo}
Let $M$ be a large constant (e.g., the maximum vertex-degree of a given instance).
Compared to formulation \maone, \matwo\ gets rid of the $\mathbf{z}$-variables with the help of big-M-type constraints. 
Formulation \matwo~ is given by:
\begin{align}
\mbox{\matwo}\quad\quad\quad w^* = &\min \sum_{i\in D} w_i x_i+ \sum_{e\in E(D)} c_e y_e \notag\\
\mbox{s.t.}\quad\quad &\mbox{\eqref{eq:dom},~\eqref{eq:wtd13} and~\eqref{eq:wtd17}}\tag{WTD2.2}\label{eq:wtd22}\\
&y_e \leq x_i + x_j - 1,\;\forall e:\{i,j\}\in E\tag{WTD2.3}\label{eq:wtd23}\\
&\sum_{e\in\delta(i)} y_e \leq 1 + M x_i,\;\forall i\in V\tag{WTD2.4}\label{eq:wtd24}\\
&\mathbf{x}\in\{0,1\}^{|V|}\;\mbox{and}\;\mathbf{y}\in\{0,1\}^{|E|} \notag
\end{align}

\paragraph{Formulation \mathree} Finally, formulation \mathree\ also gets rid of the binary $\mathbf{y}$-variables, with the help of integer variables $\mathbf{q}\in\{0,1,\ldots, |V|L\}^{|V|}$, where $L$ is a large constant, e.g., the maximum edge weight of a given instance. These variables measure for each $i \in V$ twice contribution to the objective of all the edge-weights of edges adjacent to $i$ (again, for both the internal and external edge costs).
Formulation \mathree~ is given by:
\begin{align}
\mbox{\mathree}\quad\quad\quad w^* = & \sum_{i\in V} \left(w_i x_i + \frac{1}{2}\cdot q_i\right)\tag{WTD3.1}\label{eq:wtd31}\\
\mbox{s.t.}\quad\quad & \eqref{eq:dom} \notag \\
 &q_i \geq 2\left(c_e x_i - L x_i - \sum_{e':\{i,j'\}\in E\mid_{c_{e'}\leq c_e}} Lx_{j'} \right),\;\forall e:\{i,j\}\in E,\;\forall i\in V\tag{WTD3.2}\label{eq:wtd32}\\
&q_i \geq \sum_{e:\{i,j\}\in E}c_e\left(x_i + x_j -1\right),\;\forall i\in V\tag{WTD3.3}\label{eq:wtd33}\\
&\mathbf{x}\in\{0,1\}^{|V|}\;\mbox{and}\;\mathbf{q}\in\{0,1,\ldots, |V|L\}^{|V|} \notag
\end{align}

\section{Two new Mixed-Integer Programming formulations for the \WTDP}
\label{sec:form}

In this section we present two alternative formulations that, as we show in Section \ref{sec:compres},
allow the design of algorithmic strategies that outperform the results presented in~\cite{MaEtAl2019}.

\subsection{Formulation \ffirst\ and valid inequalities}

Let $\mathbf{x}\in\{0,1\}^{|V|}$ be defined as before, 
and $\mathbf{y}\in\{0,1\}^{|E|}$ be such that $y_{e:\{i,j\}}=1$ if $x_i = 1$ and $x_j=1$,
and $y_{e:\{i,j\}}=0$, otherwise, for every edge $e:\{i,j\}\in E$.
Let $A = \{(i,j)\cup(j,i)\mid \forall e:\{i,j\}\in E \}$
be the set of bi-directed arcs associated with $E$,
and let $c_{ij} = c_{ji} = c_{e}$ for all $e\in E$.
In contrast to \maone, we associate $\mathbf{z}$ with these directed arcs, instead of the undirected edges. Let $z_{ij} = 1$ if vertex $j\in V$ is adjacent to the dominating set vertex $i$ through arc $(i,j)\in A$, and $z_{ij} =0$ otherwise. These variables are used to measure the external edge costs. 
By using such strategy, the resulting formulation resembles to the formulations of the well-known uncapacitated facility location problem (\UFL);
we can interpret the set of vertices $j \in D$ as open facilities, we want the vertices $i \in V\setminus D$ be assigned to the facility with the cheapest assignment cost (see, e.g., \cite{fischetti2016redesigning,laporte2015location} for recent references on the \UFL). Let $\delta^-(i)$ and $\delta^+(i)$ correspond to the set of \textit{incoming} and \textit{outgoing} arcs from and to vertex $i\in V$, respectively.

Using this notation, the WTDP can be formulated as follows:
\begin{align}
\mbox{\ffirst}\quad\quad\quad w^* = &\min \sum_{i\in V} w_i x_ i + \sum_{e\in E} c_e y_e + \sum_{(i,j)\in A} c_{ij} z_{ij}\notag\\
\mbox{s.t.}\quad\quad & \eqref{eq:dom} \notag \\
& x_i + \sum_{(j,i)\in\delta^-(i)} z_{ji} = 1,\;\forall i\in V\tag{XZLINK1}\label{eq:wtd42}\\
& z_{ij} \leq x_i,\;\forall (i,j)\in A \tag{XZLINK2}\label{eq:wtd43}\\
&y_e\geq x_i + x_j - 1,\;\forall e:\{i,j\}\in E\tag{YZLINK}\label{eq:wtd44}\\
&\mathbf{x}\in\{0,1\}^{|V|},\;\mathbf{y}\in\{0,1\}^{|E|}\;\mbox{and}\;\mathbf{z}\in\{0,1\}^{|A|}.  \notag
\end{align}Constraints \eqref{eq:wtd42} ensure for each $i \in V$, that either $i \in D$, or that it is covered by a $j \in D$. Constraints \eqref{eq:wtd43} link the $\mathbf{z}$-variables and $\mathbf{x}$-variables. Together with the $\sum_{(i,j)\in A} c_{ij} z_{ij}$-part of the objective function, they ensure that the contribution of vertices $i \in V \setminus D$ is measured correctly (i.e., these are the external edge costs). Finally, constraints 
\eqref{eq:wtd44} and the $\sum_{e\in E} c_e y_e$-part in the objective function make sure that the contribution of edges $e:\{i,j\}$, where both $i,j \in D$, is measured correctly (i.e., these are the internal edge costs). We note that both variable-sets $\mathbf{y}$ and $\mathbf{z}$ can be relaxed to be continuous, as for binary $\mathbf{x}$, these variables are automatically binary.

\paragraph{Valid inequalities} Next, we present three families of valid inequalities for \ffirst. Separation of these inequalities is discussed in Section \ref{sec:sep}.

\begin{theorem}
	Inequalities
	\begin{equation}
	y_{e:\{i,j\}} + z_{ij} \leq x_i,\quad \forall (i,j) \in A \label{eq:wtd43lifted}  \tag{XZLINK2L}
	\end{equation}
	
	are valid for \ffirst.
\end{theorem}

\begin{proof}
	These inequalities are a lifted version of inequalities \eqref{eq:wtd43}. Validity follows from the fact, that in any feasible solution, either the edge $e:\{i,j\}$ or the arc $(i,j)$ can be contained, and in both cases, this implies that $i \in D$.	
\end{proof}

\begin{theorem}
	Inequalities
	\begin{equation}
	\sum_{e \in\delta(i)} y_e \geq x_i,\;\forall i\in V \label{eq:domobj} \tag{TDOMY}
		\end{equation}
	are valid for \ffirst.
\end{theorem}

\begin{proof}
	By definition of total domination, for each $i \in V$, at least one adjacent vertex $j \in N(i)$ must be in $D$. Thus, if $x_i=1$, which means $i \in D$, at least one of the $y_e$-variables for $e \in \delta(i)$ must be one.
\end{proof}

%
%

For the next 
family
of valid inequalities, we observe that constraints \eqref{eq:wtd44} together with inequalities $y_{e:\{i,j\}}\leq x_i$ (which are valid, but redundant in our case due to the minimization objective) and the binary constraints on $(\mathbf{x,y})$ give the \emph{boolean quadric polytope (BQP)} (see, e.g.,~\cite{padberg1989boolean}). Thus all inequalities valid for the BQP are also valid for our formulation. We note that there are many graph problems which can be either directly formulated using the BQP, or using the BQP and additional constraints (see, e.g.,~\cite{billionnet2005different,bonomo2012polyhedral,macambira2000edge}). There is a huge number of families of valid inequalities known for the BQP, however, most of them are not useful within a branch-and-cut algorithm as there are no efficient separation procedures known for them (see, e.g.,~\cite{letchford2014new}). 
We thus just used the following inequalities known as 
\emph{clique inequalities} 
in our algorithm.

\begin{theorem}
Let $C \subset V$, such that $E(C)\subset E$ form a \emph{clique}. The	\emph{clique inequalities} 
	\begin{equation}
		\sum_{e \in E(C)} y_e  \geq \sum_{i \in C} x_i  -1 \tag{CLIQUE} \label{eq:clique}
	\end{equation}
are valid for \ffirst.
%
\end{theorem}

Section \ref{sec:sep} details how these valid inequalities are incorporated into our solution framework.

\subsection{Formulation \fsecond\ and valid inequalities}

In formulation \fsecond, we use continuous variables $q_i \geq 0$, $i \in V$ to measure the external edge costs. This is done by exploiting a Benders decomposition scheme that allows projecting out the $\mathbf{z}$-variables, similar as it is done for the \UFL~(see, e.g., \cite{fischetti2016redesigning}).
By doing so, we obtain a polynomial set of optimality cuts, which are detailed next (note that by adding a "dummy-arc"-variable $z_{ii}$ with weight zero to formulation \ffirst, replacing $x_i$ in \eqref{eq:wtd42} with $z_{ii}$ and adding a constraint $z_{ii}\leq x_i$ the connection to \UFL\ becomes directly evident; in the following, we also provide a combinatorial argument for their correctness without the need for Benders decomposition).
For ease of exposition, for a given vertex $i$, let $N'(i) = \{j_1,\ldots,j_k,\ldots, j_{|N(i)|} \}$ be the ordered set of adjacent vertices such that $c_{j_1 i}\leq \ldots \leq c_{j_k i}\leq\ldots\leq c_{j_{|N(i)|}i}$. Then the cuts for a given $i \in V$ are given by
\begin{align}\
q_i \geq c_{ki} - \sum_{k' = 1}^{k-1} (c_{ki} - c_{k'i})x_{k'} - c_{ki}x_i,\;\forall k\in \{1,\ldots,|N'(i)|\}. \label{eq:f2q} \tag{EXTCOSTS-$i$}
\end{align}
When $x_i=0$, i.e., $i \in V \setminus D$,~\eqref{eq:f2q} is similar to the Benders optimality cuts for the \UFL~and, therefore,
these inequalities measure the external edge cost for vertex $i$. 
When $x_i=1$, i.e., $i \in D$ (and thus $i$ incurs in no external edge cost), 
the right hand side of the cuts is at most zero, due to $-c_{ki}x_i$ and, therefore, they are also correct.  
By replacing~\eqref{eq:wtd42} and~\eqref{eq:wtd43} with~\eqref{eq:f2q}, the \WTDP\ can be formulated as
\begin{align}
\mbox{\fsecond}\quad\quad\quad w^* = \min &\sum_{i\in V} \left(w_i x_ i + q_i\right) + \sum_{e\in E} c_e y_e \notag\\
\mbox{s.t.}\quad\quad & \eqref{eq:dom}, \eqref{eq:f2q}, \eqref{eq:wtd44} \notag \\
&\mathbf{x}\in\{0,1\}^{|V|},\;\mathbf{y}\in\{0,1\}^{|E|}\;\mbox{and}\; q_i \geq 0, \forall i \in V. \notag
\end{align}

We note that $\mathbf{y}$-variables could also be projected out, however, the resulting optimality cuts would not have the same effective structure as \eqref{eq:f2q}.
Namely, as each $y_{e=\{i,j\}}$ links two vertices $i,j \in V$, the corresponding Benders subproblem for a fixed $\mathbf{x}$ would not decompose for each vertex.

\paragraph{Valid inequalities} Inequalities \eqref{eq:f2q} can be lifted by using the $\mathbf{y}$-variables. 

\begin{theorem}
	Let $i \in V$ and $k \in \{1,\ldots,|N'(i)|\}$. Then inequalities
\begin{align}\
q_i \geq c_{ki} - \sum_{k' = 1}^{k-1} (c_{ki} - c_{k'i})x_{k'} - c_{ki}x_i+\sum_{k' = 1}^{k-1} (c_{ki} - c_{k'i})y_{e=\{k'i\}} ,\; \label{eq:f2qlifted} \tag{EXTCOSTS-$i$-L}
\end{align}
are valid for \fsecond.
\end{theorem}	

\begin{proof}
When the $\mathbf{y}$-variables are zero, the inequalities are similar to \eqref{eq:f2q} and thus clearly valid. Now suppose some $y_{e=\{l,i\}}$ for some $1 \leq l \leq k-1$ is one. By definition of the variables, this means that both $x_i$ and $x_l$ are one, and thus on the right-hand-side (rhs) of the cut, we have $c_{ki}-(c_{ki} - c_{li}) -c_{ki}=-(c_{ki} - c_{li})<0$. Thus, $(c_{ki} - c_{li})$ (which is the coefficient of $y_e$) can be added to the rhs, which then will be zero and the inequality still remains valid. The same reasoning also applies, if more $y_e$-variables are one.
\end{proof}

Finally, we observe that inequalities~\eqref{eq:domobj} and~\eqref{eq:clique} presented for \ffirst\ are also valid for \fsecond, as they are in the $(\mathbf{x,y})$-space. 


\section{Implementation details of the branch-and-cut algorithms}
\label{sec:bc}

In this section, we give implementation details of the branch-and-cut algorithms we designed based on \ffirst\ and \fsecond. 

\subsection{Initialization and separation of cuts \label{sec:sep}}

We first describe how the valid inequalities are incorporated in our frameworks. We note that to design a successful branch-and-cut scheme, it is often crucial to carefully select which cuts to add, e.g., even if in theory the cuts improve the lower bound, they may lead to slow linear programming (LP)-relaxation solution times due to their density or numerical stability, which is detrimental to the node-throughput and thus to the overall performance of the branch-and-cut. We refer to~\citep{Dey2018,WesselmannStuhl2012, rahmaniani2017benders} for recent works on theoretical and computational studies on the challenges of cutting-plane selection. In Section~\ref{sec:ingredients} we also provide computational results obtained when just adding individual families of valid inequalities to the formulations.

The lifted inequalities~\eqref{eq:wtd43lifted} are added at the initialization, by simply replacing their non-lifted counterpart~\eqref{eq:wtd43}. The objective-cuts~\eqref{eq:f2q}, resp., their lifted version~\eqref{eq:f2qlifted} in \fsecond\ are added for the five smallest values of $c_{ki}$ for each $i\in V$ at initialization, and the remaining ones are then separated on-the-fly by enumeration. Inequalities~\eqref{eq:domobj} are also separated by enumeration.

Clique inequalities~\eqref{eq:clique} are separated heuristically. 
We observe that inequalities~\eqref{eq:wtd44} are a special case of~\eqref{eq:clique} for $|C|=2$.  
For each edge $e=\{i,j\}\in E$, we try to construct a violated inequality~\eqref{eq:clique} by greedily constructing a clique containing $e$.
Thus, initially, let $C=\{i,j\}$. Let $(\mathbf{\tilde x, \tilde y})$ be the LP-values at the current branch-and-cut node. 
We sort all vertices $k \in \cap_{i \in C} N(i)$ (i.e., all candidate vertices to grow the clique $C$) in descending order according to $|N(k)|\cdot (\tilde x_k+\epsilon)$, for $\epsilon=0.0001$.  Note that by adding any vertex $k$ to $C$, the (potential) violation of the constructed clique inequality changes by $\tilde  x_k-\sum_{i \in C} \tilde y_{e=\{i,k\}}$. Thus, we iterate through the sorted list of candidate vertices to increase $C$, and whenever this value is greater than $\epsilon$ for a given $k$, we add it to $C$, and repeat the procedure for this $C$. This is done, until no more vertex can be added to $C$. We then add the clique inequality for this $C$ if it is violated. To speed-up separation, if an edge $e$ is already contained in a clique inequality added during the current round of separation at a branch-and-cut node, we do not consider it in constructing additional clique inequalities. 

In order to avoid overloading the LP-relaxation with cuts and to allow for a fast node-throughput in the branch-and-cut, we only separate inequalities at the root node and limit separation to ten rounds. 
Naturally, to ensure correctness when using \fsecond\, violation of objective-cuts~\eqref{eq:f2q}, resp., their lifted version~\eqref{eq:f2qlifted}, is also checked whenever an integer solution is obtained during the branch-and-cut. 
As inequalities~\eqref{eq:domobj} and in particular~\eqref{eq:clique} can become quite dense, especially if the instance graph has many edges, we use the option \texttt{UseCutFilter} provided by CPLEX (the chosen MIP-solver), 
when adding these cuts. With this option, CPLEX checks the cut with the same criteria (e.g., density) as it checks its own general purpose cuts, and adds it only if it determines that it is beneficial.


\subsection{Starting and primal heuristic and local search\label{sec:heur}}

We implemented both a starting heuristic and a primal heuristic; the former gets called at the initialization while the latter gets called during the execution of the corresponding branch-and-cut algorithms. Both of these heuristics construct feasible solutions, which we then try to improve by applying a local search procedure
The starting heuristic starts out with the solution $D^H=V$ consisting of the set of all vertices (which clearly is a feasible solution). 
We then greedily remove vertices from $D^H$ as long as the solution remains feasible. 
Algorithm~\ref{alg:starting} details our starting heuristic. 

At each iteration, we use a score $score_i$ for choosing the vertex to remove;
this score gives for each vertex in $D^H$ the improvement in objective solution value it would bring if it is removed. 
When removing a vertex, say $i$, its vertex weight $v_i$ and the internal edge costs $w_{ij}$ for $j \in D^H$ are not applicable anymore. 
On the other hand, we need to consider the new external edge cost for covering $i$ and, moreover, 
we have to consider that all the vertices $j' \in V \setminus D^H$ that are covered by $i$ up to that iteration now need to be covered by another vertex in $D^H$
(thus for covering these vertices we will get new, similar or higher, external edge costs).
We note that removing a vertex only causes local changes in the solution structure, 
thus we do not need to calculate $score_i$ for each vertex in $D^H$ from scratch in every iteration. 
In particular, when node $i$ gets removed, the score needs to be re-calculated only for neighboring nodes 
$j \in N(i)$ and for the corresponding neighbors $j' \in N(j)$ 
(removing $i$ may change the external edge costs associated with such a $j'$ as both $i$ and $j'$ share $j$ as neighbor). 
We observe that verifying if $D^H$ is still be a total dominating set after removing $i$ 
(line~\ref{alg:check} of Algorithm~\ref{alg:starting}) 
can be done efficiently by storing the number $N^H_j=|N(j) \cap D^H|$ for each $j \in V$, 
i.e., the number of neighbors of $j$ contained in $D^H$. 
At the begging of the algorithm execution, it holds $N^H_j=|N(j)|$, and whenever a vertex $i'$ gets removed in the course of the algorithm, $N^H_j$ gets decreased by one for each $j \in N(i)$. 
Therefore, $D^H \setminus \{i\}$ is still a total dominating set, if and only if $N^H_j>1$ for each $j \in N(i)$. 


\SetKwRepeat{Do}{do}{while}
\SetKw{Continue}{continue}
\begin{algorithm}[h!tb]   
	\DontPrintSemicolon                 
	\SetKwInOut{Input}{input}\SetKwInOut{Output}{output}
	\Input{instance $(G=(V,E), \mathbf{(c,w)})$ of the \WTDP}
	\Output{total dominating set $D^H$}
	$D^H\gets V$ \label{alg:start}\;
 	$score_i \gets -\infty$ \tcp*[f]{store score function for faster evaluation, $-\infty$ indicates re-calculation needed} \;
 	$improvingMoveExists \gets false$ \;
	\Do{$improvingMoveExists$}
	{
		$improvingMoveExists \gets false$ \;
		$vertexToRemove \gets null$\;
		$bestScore=0$\;
		\For{$i \in D^H$}
		{
			 \If{$D^H \setminus \{i\}$ is not a total dominating set\label{alg:check}}{\Continue}
			\If(\tcp*[f]{re-calculation of score needed}){$score_i=-\infty$}
			{
				$score_i\gets w_i$\tcp*[f]{vertex costs saved}\;
				\For{$j \in N(i)\cap D^H$}
				{
					$score_i\gets score_i+w_{ij}$\tcp*[f]{internal edge costs saved}\;	
				}
				$w^* \gets \min_{j \in D^H} w_{ij}$\tcp*[f]{external edge cost for covering $i$ updated}\;
				$score_i\gets score_i -w^*$\;
				\For(\tcp*[f]{external edge costs for vertices currently covered by $i$ updated}){$j \in N(i)\cap (V \setminus D^H)$}
				{
					\If{$i \in \argmin_{j'\in D^H} w_{jj'}$}
					{
						$w_j^*\gets \min_{j \in (D^H \setminus \{i\})} w_{ij}$\tcp*[f]{external edge cost for covering $j$ updated}\;
						$score_i\gets score_i-w^*_{j}$\;
					}
				}
				
			}
		
			\If{$score_i>bestScore$}
			{
				$bestScore=score_i$ \label{alg:upstart}\;
				$vertexToRemove \gets i$\;
				$improvingMoveExists \gets true$ \label{alg:upend}\;	
			}
		}
		\If{$improvingMoveExists$}
		{
			$D^H \gets D^H \setminus \{vertexToRemove\}$\;
			\For(\tcp*[f]{score for the neigbors of $vertexToRemove$, and their neigbors need to be updated}){$\forall j \in N(vertexToRemove)$}
			{
			\For{$\forall j' \in N(j)$}
			{
					$score_{j'}\gets -\infty$\;
			}	
			}
		}
	}
	\caption{Starting heuristic\label{alg:starting} }
\end{algorithm}


The primal heuristic is guided by the $\mathbf{(\tilde x)}$-values of the LP-relaxation at the current branch-and-cut node. First, we sort the vertices $i \in V$ in descending order according to $\tilde x_i$. 
Afterwards, ties are broken first by degree of the vertices (again in descending order), and if there remain ties, they are broken by vertex-index. Let $sorted$ be the list of sorted vertices, $D^H=\emptyset$ (the solution to be constructed) and $covered=\emptyset$ (the list of vertices covered by $D^H$). 
To construct a heuristic solution, we iterate through $sorted$ and whenever $|N(i) \cap (V\setminus covered)|>0$ for the currently considered vertex $i$, i.e., $i$ covers a vertex not yet covered by the current partial solution $D^H$, we add $i$ to $D^H$ and update $covered$ by $covered \cup N(i)$. We stop when $covered=V$, i.e., $D^H$ is a total dominating set and thus a feasible solution. 

The local search procedure is shown in Algorithm~\ref{alg:local}. It uses two local search operators, namely adding a vertex $i$ to the current solution $D^H$ and removing a vertex $i$ from the current solution $D^H$. The procedures \texttt{testAddVertex($D^H$,$i$)} and \texttt{testRemoveVertex($D^H$,$i$)} revert the change in objective function caused by adding/removing a vertex $i$. This can be done efficiently, as the changes caused by these moves are of a local nature, as described above (e.g., the test for the change caused by removing $i$ is exactly the calculation of the score-function in Algorithm~\ref{alg:starting}). We first try the add-move, and when this move cannot improve the current solution anymore, we try the remove-move. If it is successful, we go back to trying the add-move, if not, the local search terminates. 
We iterate through the vertices by their indexes, and if a move is possible, we apply it, and then restart (i.e., we use a \emph{first improvement} strategy).

\SetKwRepeat{Do}{do}{while}
\SetKw{Continue}{continue}
\begin{algorithm}[h!tb]   
	\DontPrintSemicolon                 
	\SetKwInOut{Input}{input}\SetKwInOut{Output}{output}
	\Input{total dominating set $D^H$}
	\Output{total dominating set $D^H$ (with potentially better objective function value)}
	$improvingMoveExists \gets false$ \;
	\Do{$improvingMoveExists$}
	{
			$improvingMoveExists \gets false$ \;
	\For{$i \not \in D^H$}
	{
		\If{\texttt{testAddVertex($D^H$,$i$)}$>0$}
		{
			$D^H=D^H \cup\{i\}$\;
			$improvingMoveExists \gets true$ \;
			\bf{break}\;
		}
	}
	\If{$improvingMoveExists==false$}
		{
		\For{$i \in D^H$}
		{
			\If{\texttt{testRemoveVertex($D^H$,$i$)}$>0$}
			{
				$D^H=D^H \setminus \{i\}$\;
				$improvingMoveExists \gets true$ \;
				\bf{break}\;
			}
			
		}
	}

}	
	\caption{Local search\label{alg:local} }
\end{algorithm}

\subsection{Branching priorities \label{sec:bra}}

For both \ffirst\ and \fsecond, once the $\mathbf{x}$-variables are fixed to binary, the values of all the other variables (i.e., $\mathbf{(y,z)}$, resp., $\mathbf{(y,q)}$) automatically follow. We thus give branching priorities $100\cdot|N(i)|$ to the $\mathbf{x}$-variables in the MIP-solver, CPLEX in our case (while the branching priorities of the other variables were left at their default value, i.e., zero).

\section{A genetic algorithm \label{sec:genalg}}

Genetic algorithms (GAs) are among the most prominent metaheuristic approaches for solving (combinatorial) optimization problems; we refer the reader to the book~\cite{Kramer2017} for an overview on essential elements of this class of
procedures.
GAs have been developed for tackling set dominating problems.
For instance, a hybrid GA has been developed in~\cite{hedar2010hybrid} for the \emph{minimum dominating set problem}, 
where the GA methodology is combined with local search and intensification schemes;
likewise, in~\cite{giap2014parallel} a parallelized GA is presented for the same problem.
Further examples on GA-based approaches for related problems can be found
in~\cite{sundar2014steady} for the \emph{dominating tree problem},
and in~\cite{RengaswamyEtAl2017} for the \emph{minimum weight minimum connected dominating set problem}.

In their general setting, GAs explore the solution space by keeping a set of feasible solutions, denoted as \emph{population}.
Starting from an initial population, the algorithm iteratively creates a new population (i.e., new solutions) by typically using the following three (randomized) bio-inspired operators: \emph{selection}, \emph{mutation} and \emph{crossover}. 
The 
\emph{selection} operator selects a subset of the current population (according to a \emph{fitness} value of each solution), from which (usually) pairs of solutions are taken and a \emph{crossover} operator is applied to combine these pairs to create a new solutions. 
To these new solutions a \emph{mutation} operator is applied, which randomly modifies the solution in order to keep the population diverse.


Algorithm~\ref{alg:genetic} gives an outline of the genetic algorithm we developed for the \WTDP.
The initial population is constructed by using a generalized randomized adaptive search procedure (GRASP) version of our starting heuristic. 
GRASP is a general technique to generate (diverse) heuristic solutions by randomizing the construction phase. 
For further details on GRASP, the reader is referred to the recent textbook~\cite{ResendeRibeiro2016}. 
In order to turn our starting heuristic into a GRASP, we add randomization to the choosing of the vertex with the best score (i.e., lines \ref{alg:upstart}-\ref{alg:upend}): 
If a vertex $i$ has $score_i>bestScore$, we generate a random integer in $[0,99]$ and only apply lines \ref{alg:upstart}-\ref{alg:upend} if this integer is larger than a given value $cutoff$.

As \emph{crossover} operator, we also use modifications of Algorithm~\ref{alg:starting}, resp., the GRASP. 
In particular, for crossover between two solution $D^1$, $D^2$, we use the GRASP, and set $D^H \gets D^1 \cup D^2$ as initial solution in line~\ref{alg:start}. 
For \emph{mutation}, we generate a random integer $m$ in a given range $[m_l,m_u]$ and then randomly remove $m$ vertices from the current solution $D^H$. 
After removing these vertices, $D^H$ may be infeasible, in order to make it feasible, we apply the same heuristic as our primal heuristic (with just the degree of vertices as sorting criterion, as of course we have no LP-values). 
After mutation, we also apply the local search procedure described in Algorithm~\ref{alg:local}. 
The newly obtained solutions are merged with the current population, and then the $populationSize$ best are selected as the next generation, for a given value of $populationSize$. 
As a fitness value for selection, we use the objective function values of the solutions. In order to keep the population diverse, we keep at most one solution for each fitness value and size $|D^H|$ in the population (this is done by checking if the current population already contains a solution with the fitness value and size of the currently created solution, and if yes, the solution is discarded).
To create the population, we run the GRASP $initialPopulationSize$ times, for a given value of parameter $initialPopulationSize$ and then select the $populationSize$ best solutions. 

We used the following parameter values in our implementation, these values were determined using some preliminary computational experiments: $initialPopulationSize=100$, $populationSize=40$, $cutoff=30$, $[m_l,m_u]=[1,4]$, and $nIterations=20$.

\SetKwRepeat{Do}{do}{while}
\SetKw{Continue}{continue}
\begin{algorithm}[h!tb]   
	\DontPrintSemicolon                 
	\SetKwInOut{Input}{input}\SetKwInOut{Output}{output}
	\Input{instance $I=(G=(V,E), \mathbf{(c,w)})$ of the \WTDP, parameters $initialPopulationSize, populationSize, cutoff, [m_l,m_u], nIterations$}
	\Output{total dominating set $D^H$}
	$population \gets \emptyset$\label{alg:grasp1} \;
\For{$i=1,\ldots, initialPopulationSize$}
{
	$newD \gets \texttt{GRASP}(I, cutoff)$\;
	\If{there is no solution with the same objective value and size as $newD$ in $population$}
	{
		$population \gets population \cup newD$\label{alg:grasp2} \;	
	}
}
$population \gets \texttt{select}(population,populationSize)$ \;
\For{$i=1,\ldots, nIterations$}
{
\For{all pairs $D^1,D^2$ from $population$}
{
$newD \gets \texttt{crossover}(D^1,D^2, cutoff)$\;
$newD \gets \texttt{mutation}(newD,[m_l,m_u])$\;
$newD \gets \texttt{localSearch}(newD)$\;
\If{there is no solution with the same objective value and size as $newD$ in $population$}
{
$population \gets population \cup newD$ \;	
}
}
$population \gets \texttt{select}(population,populationSize)$\;
}
$D^H\gets \texttt{select}(population,1)$\;
	\caption{Genetic algorithm\label{alg:genetic} }
\end{algorithm}

\section{Computational results \label{sec:compres}}

The branch-and-cut framework was implemented in C++ using CPLEX 12.9 as MIP solver and the genetic algorithm was also implemented in C++.
The computational study was carried out on an Intel Xeon E5 v4 CPU with 2.5 GHz and 6GB memory using a single thread. All CPLEX parameters were left at default values (except branching priorities, see Section~\ref{sec:bra}), and we set the timelimit for a run to 1800 seconds (similar to the timelimit in~\cite{MaEtAl2019}). 

\subsection{Instance description}
\label{subsec:instdesc}

\paragraph{Instances similar to the instances of~\cite{MaEtAl2019}} 
In~\cite{MaEtAl2019}, the authors created instances to test their formulations. 
Unfortunately, the instances are not available online, thus we generated our own, following the same procedure as described in~\cite{MaEtAl2019}.
These instances are generated according to the 
Erd\"os-R\'enyi model, where one fixes the number of nodes, $|V|$ and a probability $p\in[01]$ that allows to control  the edge density of the resulting graph. Edge and vertex weights,
$\mathbf{c}$ and $\mathbf{w}$, respectively, are random integers between one and five. 
As in~\cite{MaEtAl2019}, we considered $|V|\in\{20,50,100\}$ and $p\in\{0.2,0.5,0.8\}$,
which leads to instance ranging from $20$ nodes and $31$ edges to $100$ nodes and $3943$ edges.
For each pair $(n,p)$ we generated five instances (instead of one as done in~\cite{MaEtAl2019}), using the \texttt{gnp\_random\_graph(n,p)}-method from the \texttt{networkx}-package~\cite{hagberg2008exploring} to obtain the Erd\"os-R\'enyi graphs. This set has $3 \cdot 3 \cdot 5 = 45$ instances.
We denote this set of instances as \MAINST, individual instances are addressed as \MAINST$-|V|-p-id$, where $id \in \{1,\ldots, 5\}$.

\paragraph{New instances} 
To analyze the influence of different weight structures, we generated an additional set of instances, denoted as \NEWINST, as in the \MAINST\, both are in a similar (small) range.
%
We again used the Erd\"os-R\'enyi model, and considered $|V|\in\{75,100,125\}$ and $p\in\{0.2,0.5,0.8\}$. 
We used the following range-combinations for $(\mathbf{c},\mathbf{w})$: $([1,50],[1,10])$, $([1,25],[1,25])$,$([1,10],[1,50])$. 
For each combination of $(|V|,p)$ and $(\mathbf{c},\mathbf{w})$ we created five instances. 
Thus, this set has $3 \cdot 3 \cdot 3 \cdot 5 = 135$ instances.
The instance set is denoted as \NEWINST, individual instances are addressed as \NEWINST$-|V|-p-c_u-id$, 
where $id \in \{1,\ldots, 5\}$ and $c_u$ is the upper bound of the considered range for $\mathbf{c}$ (i.e., $c_u \in \{10,25,50\}$). 

Both sets of instances we created are available online at \url{https://msinnl.github.io/pages/instancescodes.html}\sloppy.

%
%

\subsection{Assessing the effect of the valid inequalities \label{sec:ingredients}}

We now analyze the effect of the families of valid inequalities presented in Section~\ref{sec:form}.
In order to test this, we added them individually to the corresponding model of the LP-relaxation of \ffirst\ and \fsecond, and then also added all of them together.
There is an exponential number of cliques; therefore, in order to get an impression of the effect of clique inequalities~\eqref{eq:clique},  we heuristically calculate edge clique covers (i.e., set of cliques, such that every edge occurs in one of the cliques) using our separation heuristic described in Section~\ref{sec:sep} (using the vertex-degree as sorting criteria), and add the corresponding inequalities induced by the cliques in this cover. 

In figure~\ref{fig:lpplot3}, we give a plot of the obtained LP gaps (over both instance sets), calculated as $100 \cdot (w^B-w^{LP})/w^B$, where $w^B$ is the best solution value we obtained using our approaches, and $w^{LP}$ is the value of the considered LP relaxation. Note that the figure does not give a plot for \ffirst+\eqref{eq:wtd43lifted},\fsecond,\fsecond+\eqref{eq:f2qlifted},\fsecond+\eqref{eq:domobj}, \fsecond+\eqref{eq:clique}. This is, because using the liftings~\eqref{eq:wtd43lifted}, resp.,~\eqref{eq:f2qlifted} on their own had no effect on the value of the LP relaxation. Moreover, the values obtained by \fsecond,\fsecond+\eqref{eq:domobj}, \fsecond+\eqref{eq:clique} are just the same as their counterpart with \ffirst, as in \fsecond, the $\mathbf{z}$-variables are projected out in a Benders way, and~\eqref{eq:domobj} and~\eqref{eq:clique} operate in the $(\mathbf{x},\mathbf{y})$-space. 
On the other hand, there is a slight difference between \ffirst+all and \fsecond+all, with \ffirst+all giving slightly lower gaps. An explanation for this is, that once inequalities~\eqref{eq:domobj} and~\eqref{eq:clique} are present in the model, the liftings~\eqref{eq:wtd43lifted}, resp.,~\eqref{eq:f2qlifted} also start to have an effect on the bounds (as both~\eqref{eq:domobj} and~\eqref{eq:clique} "push" the $\mathbf{y}$-variables, which are the variables added in the lifting).

Overall, we see that both~\eqref{eq:domobj} and~\eqref{eq:clique} on their own result in a considerable improvement of the LP gap, e.g., without any inequalities only around 20\% of the instances have an LP gap of 20\% or less, while adding~\eqref{eq:domobj} or~\eqref{eq:clique} increases the number of instances with such a gap to about 30\%. 
Adding all families together gives again a considerable improvement, now for about 50\% of instances, the LP gap is 20\% or less. When adding all inequalities, the largest LP gap is around 50\%, while without adding any inequalities, the largest gap is over 70\%.

\begin{figure}[h!tb]
%
				\centering \includegraphics[width=.80\linewidth]{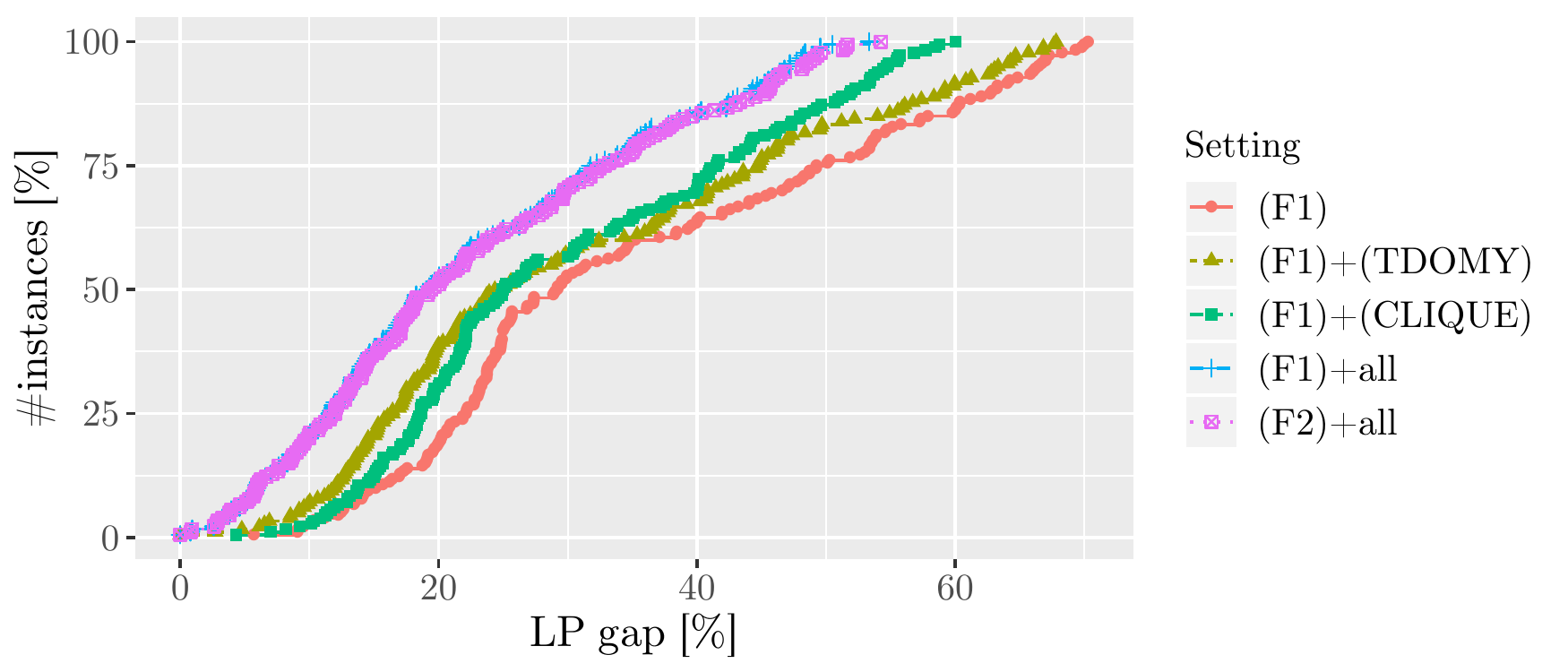}
	\caption{LP-gap plot for adding families of valid inequalities to \ffirst\ and \fsecond.}\label{fig:lpplot3}	
\end{figure}


\subsection{Detailed results \label{sec:detailed}}

In this section, we provide detailed results obtained by the branch-and-cut frameworks based on the formulations \ffirst\ and \fsecond\ described in Section~\ref{sec:bc}, and also the genetic algorithm. 
A comparison with the models of~\cite{MaEtAl2019} is also done. 

In Table~\ref{ta:comp} we report the following results for instance set \MAINST.
In columns \ffirst+ and \fsecond+, we report the results attained by our branch-and-cut algorithms.
In columns \ffirst\ and \fsecond\, we report the  results attained 
when solving \ffirst\ and \fsecond\ directly with CPLEX, without any of the branch-and-cut ingredients presented in Section~\ref{sec:bc}  (note that in case of \fsecond, constraints~\ref{eq:f2q} are separated as they are needed to ensure correctness).
In columns \maone, \matwo, and \mathree\,
we report the results attained when solving directly by CPLEX
the formulations provided in~\cite{MaEtAl2019}.
In this table, we report for each approach the runtime (column $t[s]$), the objective value of the best obtained solution (column $w^B$) and the optimality gap (column $g[\%]$, calculated as $100\cdot(w^B-LB)/w^B$, where $LB$ is the obtained lower bound).

The results reported in Table~\ref{ta:comp} show,
that the approaches proposed in this paper (i.e., \ffirst, \fsecond, \ffirst+ and \fsecond+),
are considerably more effective than those proposed in~\cite{MaEtAl2019}:
All of our approaches, with the exception of \ffirst, managed to solve all the instances within the timelimit, while none of the approaches  \maone, \matwo, and \mathree\ managed to solve the instances with 100 nodes to optimality (and \maone, \matwo\ also fail for some of the instances with 50 nodes).
Moreover, for the instances where our approaches as well as the approaches of~\cite{MaEtAl2019} can solve the problem,
our approaches are up to 500 times faster; see, e.g, instances \texttt{MA-50-0.5-4}, where \fsecond+ takes 2 seconds, while \mathree\ takes 1201 seconds (and \maone, \matwo\ reach the timelimit). 
Likewise,the gaps attained by our strategies (when the time limit is reached),
are considerable smaller than those attained by the~\cite{MaEtAl2019} approaches 
(for instance, while the maximum gap attained by \ffirst\ is 9.06\%, the maximum gap attained by \mathree\ is 72.53\%).
When comparing among our approaches, we see that for all but two instances (\texttt{MA-100-0.8-3} and \texttt{MA-100-0.8-4}), \fsecond+ is the fastest, and for these two instances \fsecond\ is the fastest. 
Not surprisingly, the instances become harder with larger number of vertices, and, also higher density ($p$) seems to make the instances harder.



\setlength{\tabcolsep}{3pt}
\begin{table}[t!]
	\centering
	\caption{Comparison with previous approaches \maone, \matwo, \mathree\ from literature on instance set \MAINST \label{ta:comp}} 
	\begingroup\footnotesize
	\begin{tabular}{rrr|rrrrrrrrrrrr|rrrrrrrrr}
		\toprule
		\multicolumn{3}{c}{instance} & \multicolumn{3}{c}{\ffirst} & \multicolumn{3}{c}{\ffirst+} & \multicolumn{3}{c}{\fsecond} & \multicolumn{3}{c}{\fsecond+} & \multicolumn{3}{c}{\maone} & \multicolumn{3}{c}{\matwo} & \multicolumn{3}{c}{\mathree}  \\ $|V|$ & $p$ & $id$ & $t[s]$ & $w^*$ & g[\%] & $t[s]$ & $w^*$ & g[\%]& $t[s]$ & $w^*$ & g[\%]& $t[s]$ & $w^*$ & g[\%]& $t[s]$ & $w^*$ & g[\%]& $t[s]$ & $w^*$ & g[\%]& $t[s]$ & $w^*$ & g[\%] \\ \midrule
		20 & 0.2 & 1 & \textbf{1} & 63 & 0.00 & \textbf{1} & 63 & 0.00 & \textbf{1} & 63 & 0.00 & \textbf{1} & 63 & 0.00 & 2 & 63 & 0.00 & 3 & 63 & 0.00 & 2 & 63 & 0.00 \\ 
		20 & 0.2 & 2 & \textbf{1} & 58 & 0.00 & \textbf{1} & 58 & 0.00 & \textbf{1} & 58 & 0.00 & \textbf{1} & 58 & 0.00 & 3 & 58 & 0.00 & 3 & 58 & 0.00 & \textbf{1} & 58 & 0.00 \\ 
		20 & 0.2 & 3 & 3 & 58 & 0.00 & \textbf{1} & 58 & 0.00 & \textbf{1} & 58 & 0.00 & \textbf{1} & 58 & 0.00 & 3 & 58 & 0.00 & 3 & 58 & 0.00 & \textbf{1} & 58 & 0.00 \\ 
		20 & 0.2 & 4 & \textbf{1} & 51 & 0.00 & \textbf{1} & 51 & 0.00 & \textbf{1} & 51 & 0.00 & \textbf{1} & 51 & 0.00 & 4 & 51 & 0.00 & 3 & 51 & 0.00 & 2 & 51 & 0.00 \\ 
		20 & 0.2 & 5 & \textbf{1} & 55 & 0.00 & \textbf{1} & 55 & 0.00 & \textbf{1} & 55 & 0.00 & \textbf{1} & 55 & 0.00 & 3 & 55 & 0.00 & 4 & 55 & 0.00 & \textbf{1} & 55 & 0.00 \\ 
		20 & 0.5 & 1 & \textbf{1} & 44 & 0.00 & \textbf{1} & 44 & 0.00 & \textbf{1} & 44 & 0.00 & \textbf{1} & 44 & 0.00 & 5 & 44 & 0.00 & 4 & 44 & 0.00 & \textbf{1} & 44 & 0.00 \\ 
		20 & 0.5 & 2 & \textbf{1} & 47 & 0.00 & \textbf{1} & 47 & 0.00 & \textbf{1} & 47 & 0.00 & \textbf{1} & 47 & 0.00 & 4 & 47 & 0.00 & 5 & 47 & 0.00 & \textbf{1} & 47 & 0.00 \\ 
		20 & 0.5 & 3 & \textbf{1} & 46 & 0.00 & \textbf{1} & 46 & 0.00 & \textbf{1} & 46 & 0.00 & \textbf{1} & 46 & 0.00 & 6 & 46 & 0.00 & 4 & 46 & 0.00 & \textbf{1} & 46 & 0.00 \\ 
		20 & 0.5 & 4 & \textbf{1} & 40 & 0.00 & \textbf{1} & 40 & 0.00 & \textbf{1} & 40 & 0.00 & \textbf{1} & 40 & 0.00 & 4 & 40 & 0.00 & 3 & 40 & 0.00 & \textbf{1} & 40 & 0.00 \\ 
		20 & 0.5 & 5 & \textbf{1} & 41 & 0.00 & \textbf{1} & 41 & 0.00 & \textbf{1} & 41 & 0.00 & \textbf{1} & 41 & 0.00 & 4 & 41 & 0.00 & 3 & 41 & 0.00 & \textbf{1} & 41 & 0.00 \\ 
		20 & 0.8 & 1 & \textbf{1} & 37 & 0.00 & \textbf{1} & 37 & 0.00 & \textbf{1} & 37 & 0.00 & \textbf{1} & 37 & 0.00 & 3 & 37 & 0.00 & 4 & 37 & 0.00 & 2 & 37 & 0.00 \\ 
		20 & 0.8 & 2 & 3 & 35 & 0.00 & \textbf{1} & 35 & 0.00 & \textbf{1} & 35 & 0.00 & \textbf{1} & 35 & 0.00 & 4 & 35 & 0.00 & 4 & 35 & 0.00 & \textbf{1} & 35 & 0.00 \\ 
		20 & 0.8 & 3 & \textbf{1} & 40 & 0.00 & \textbf{1} & 40 & 0.00 & \textbf{1} & 40 & 0.00 & \textbf{1} & 40 & 0.00 & 5 & 40 & 0.00 & 4 & 40 & 0.00 & \textbf{1} & 40 & 0.00 \\ 
		20 & 0.8 & 4 & \textbf{1} & 34 & 0.00 & \textbf{1} & 34 & 0.00 & \textbf{1} & 34 & 0.00 & \textbf{1} & 34 & 0.00 & 4 & 34 & 0.00 & 5 & 34 & 0.00 & \textbf{1} & 34 & 0.00 \\ 
		20 & 0.8 & 5 & \textbf{1} & 34 & 0.00 & \textbf{1} & 34 & 0.00 & \textbf{1} & 34 & 0.00 & \textbf{1} & 34 & 0.00 & 5 & 34 & 0.00 & 3 & 34 & 0.00 & \textbf{1} & 34 & 0.00 \\ 
		\midrule
		50 & 0.2 & 1 & 2 & 111 & 0.00 & \textbf{1} & 111 & 0.00 & 4 & 111 & 0.00 & \textbf{1} & 111 & 0.00 & 991 & 111 & 0.00 & 216 & 111 & 0.00 & 229 & 111 & 0.00 \\ 
		50 & 0.2 & 2 & 3 & 106 & 0.00 & \textbf{1} & 106 & 0.00 & 3 & 106 & 0.00 & \textbf{1} & 106 & 0.00 & 1380 & 106 & 0.00 & 438 & 106 & 0.00 & 309 & 106 & 0.00 \\ 
		50 & 0.2 & 3 & 8 & 111 & 0.00 & \textbf{1} & 111 & 0.00 & 4 & 111 & 0.00 & \textbf{1} & 111 & 0.00 & TL & 114 & 10.40 & 755 & 111 & 0.00 & 552 & 111 & 0.00 \\ 
		50 & 0.2 & 4 & 4 & 101 & 0.00 & \textbf{1} & 101 & 0.00 & 4 & 101 & 0.00 & \textbf{1} & 101 & 0.00 & 796 & 101 & 0.00 & 322 & 101 & 0.00 & 409 & 101 & 0.00 \\ 
		50 & 0.2 & 5 & 13 & 108 & 0.00 & \textbf{1} & 108 & 0.00 & 6 & 108 & 0.00 & \textbf{1} & 108 & 0.00 & TL & 108 & 12.19 & 1565 & 108 & 0.00 & 1129 & 108 & 0.00 \\ 
		50 & 0.5 & 1 & 4 & 82 & 0.00 & 3 & 82 & 0.00 & 3 & 82 & 0.00 & \textbf{2} & 82 & 0.00 & TL & 82 & 19.75 & TL & 82 & 9.11 & 631 & 82 & 0.00 \\ 
		50 & 0.5 & 2 & 5 & 85 & 0.00 & \textbf{2} & 85 & 0.00 & 3 & 85 & 0.00 & \textbf{2} & 85 & 0.00 & TL & 88 & 19.39 & 1579 & 85 & 0.00 & 794 & 85 & 0.00 \\ 
		50 & 0.5 & 3 & 22 & 84 & 0.00 & 3 & 84 & 0.00 & 4 & 84 & 0.00 & \textbf{2} & 84 & 0.00 & TL & 87 & 18.31 & TL & 85 & 9.37 & 1082 & 84 & 0.00 \\ 
		50 & 0.5 & 4 & 16 & 82 & 0.00 & 3 & 82 & 0.00 & 4 & 82 & 0.00 & \textbf{2} & 82 & 0.00 & TL & 82 & 17.55 & TL & 82 & 14.90 & 1201 & 82 & 0.00 \\ 
		50 & 0.5 & 5 & 15 & 82 & 0.00 & 4 & 82 & 0.00 & 4 & 82 & 0.00 & \textbf{2} & 82 & 0.00 & TL & 83 & 20.30 & TL & 82 & 12.69 & 1062 & 82 & 0.00 \\ 
		50 & 0.8 & 1 & 6 & 77 & 0.00 & 7 & 77 & 0.00 & 5 & 77 & 0.00 & \textbf{4} & 77 & 0.00 & TL & 77 & 8.39 & 1063 & 77 & 0.00 & 876 & 77 & 0.00 \\ 
		50 & 0.8 & 2 & 3 & 72 & 0.00 & 3 & 72 & 0.00 & 3 & 72 & 0.00 & \textbf{2} & 72 & 0.00 & 452 & 72 & 0.00 & 307 & 72 & 0.00 & 299 & 72 & 0.00 \\ 
		50 & 0.8 & 3 & 3 & 74 & 0.00 & 3 & 74 & 0.00 & 4 & 74 & 0.00 & \textbf{2} & 74 & 0.00 & 1100 & 74 & 0.00 & 587 & 74 & 0.00 & 446 & 74 & 0.00 \\ 
		50 & 0.8 & 4 & 6 & 76 & 0.00 & 5 & 76 & 0.00 & 4 & 76 & 0.00 & \textbf{3} & 76 & 0.00 & TL & 76 & 10.23 & 1242 & 76 & 0.00 & 736 & 76 & 0.00 \\ 
		50 & 0.8 & 5 & 13 & 79 & 0.00 & 15 & 79 & 0.00 & \textbf{7} & 79 & 0.00 & \textbf{7} & 79 & 0.00 & TL & 79 & 16.13 & 1720 & 79 & 0.00 & 1310 & 79 & 0.00 \\ 
			\midrule
		100 & 0.2 & 1 & 898 & 175 & 0.00 & 92 & 175 & 0.00 & 566 & 175 & 0.00 & \textbf{50} & 175 & 0.00 & TL & 183 & 38.79 & TL & 178 & 34.82 & TL & 175 & 56.48 \\ 
		100 & 0.2 & 2 & 251 & 174 & 0.00 & 14 & 174 & 0.00 & 276 & 174 & 0.00 & \textbf{12} & 174 & 0.00 & TL & 174 & 34.17 & TL & 175 & 34.60 & TL & 188 & 61.21 \\ 
		100 & 0.2 & 3 & TL & 178 & 6.48 & 239 & 177 & 0.00 & 1434 & 177 & 0.00 & \textbf{121} & 177 & 0.00 & TL & 195 & 43.27 & TL & 189 & 41.47 & TL & 183 & 60.95 \\ 
		100 & 0.2 & 4 & TL & 169 & 2.17 & 81 & 169 & 0.00 & 562 & 169 & 0.00 & \textbf{37} & 169 & 0.00 & TL & 172 & 37.79 & TL & 175 & 36.77 & TL & 172 & 59.06 \\ 
		100 & 0.2 & 5 & TL & 171 & 5.39 & 97 & 167 & 0.00 & 1473 & 167 & 0.00 & \textbf{47} & 167 & 0.00 & TL & 170 & 39.18 & TL & 172 & 38.20 & TL & 173 & 60.06 \\ 
		100 & 0.5 & 1 & TL & 147 & 2.69 & 304 & 147 & 0.00 & 292 & 147 & 0.00 & \textbf{108} & 147 & 0.00 & TL & 166 & 51.92 & TL & 160 & 49.55 & TL & 149 & 62.37 \\ 
		100 & 0.5 & 2 & 707 & 144 & 0.00 & 158 & 144 & 0.00 & 152 & 144 & 0.00 & \textbf{51} & 144 & 0.00 & TL & 154 & 44.81 & TL & 148 & 42.81 & TL & 150 & 66.31 \\ 
		100 & 0.5 & 3 & 995 & 147 & 0.00 & 401 & 147 & 0.00 & 186 & 147 & 0.00 & \textbf{128} & 147 & 0.00 & TL & 157 & 48.58 & TL & 149 & 43.22 & TL & 160 & 62.88 \\ 
		100 & 0.5 & 4 & TL & 149 & 9.06 & 725 & 146 & 0.00 & 289 & 146 & 0.00 & \textbf{214} & 146 & 0.00 & TL & 156 & 50.99 & TL & 150 & 46.05 & TL & 160 & 63.98 \\ 
		100 & 0.5 & 5 & TL & 139 & 3.18 & 466 & 139 & 0.00 & 242 & 139 & 0.00 & \textbf{155} & 139 & 0.00 & TL & 148 & 49.27 & TL & 145 & 44.48 & TL & 152 & 71.38 \\ 
		100 & 0.8 & 1 & 1655 & 136 & 0.00 & 346 & 136 & 0.00 & 172 & 136 & 0.00 & \textbf{97} & 136 & 0.00 & TL & 150 & 55.44 & TL & 141 & 51.52 & TL & 136 & 62.64 \\ 
		100 & 0.8 & 2 & 759 & 140 & 0.00 & 894 & 140 & 0.00 & 283 & 140 & 0.00 & \textbf{249} & 140 & 0.00 & TL & 146 & 49.55 & TL & 141 & 44.35 & TL & 147 & 65.31 \\ 
		100 & 0.8 & 3 & 1212 & 141 & 0.00 & 1032 & 141 & 0.00 & \textbf{236} & 141 & 0.00 & 325 & 141 & 0.00 & TL & 144 & 53.13 & TL & 149 & 50.87 & TL & 153 & 71.51 \\ 
		100 & 0.8 & 4 & TL & 141 & 7.45 & 1652 & 141 & 0.00 & \textbf{334} & 141 & 0.00 & 495 & 141 & 0.00 & TL & 148 & 52.31 & TL & 147 & 50.83 & TL & 142 & 71.17 \\ 
		100 & 0.8 & 5 & 990 & 134 & 0.00 & 509 & 134 & 0.00 & 231 & 134 & 0.00 & \textbf{160} & 134 & 0.00 & TL & 152 & 55.60 & TL & 148 & 51.19 & TL & 156 & 72.53 \\ 
		\bottomrule
	\end{tabular}
	\endgroup
\end{table}

\pagebreak

In the following, we focus on the instance set \NEWINST\ and our solution algorithms to get more insights on the performance of them, in particular, their behavior with respect to different weight structures. 
In Figure~\ref{fig:ilpplot}, we present plots of runtimes to optimality, and optimality gaps (for the unsolved instances of the respective approaches) for \ffirst\, \fsecond, \ffirst+ and \fsecond+. 
Figures~\ref{fig:ilpplot1} and~\ref{fig:ilpplot2} give runtimes optimality gaps, respectively, 
for the complete set of instances \NEWINST, while Figures~\ref{fig:ilpplot3}-\ref{fig:ilpplot7} 
show results for the different weight structure, i.e., 
Figures~\ref{fig:ilpplot3} and~\ref{fig:ilpplot4} are for the instances with $c_u=10$, 
Figures~\ref{fig:ilpplot5} and~\ref{fig:ilpplot6} are for the instances with $c_u=25$ and 
Figure~\ref{fig:ilpplot7} is for the instances with $c_u=50$ (since all instances are solved to optimality we do not provide an optimality gap plot). 
Note the different scales on the x-axis of Figure~\ref{fig:ilpplot7} compared to the other runtime-plots.

\begin{figure}[h!]
	\begin{subfigure}[b]{.5\linewidth}
		\centering \includegraphics[width=.99\linewidth]{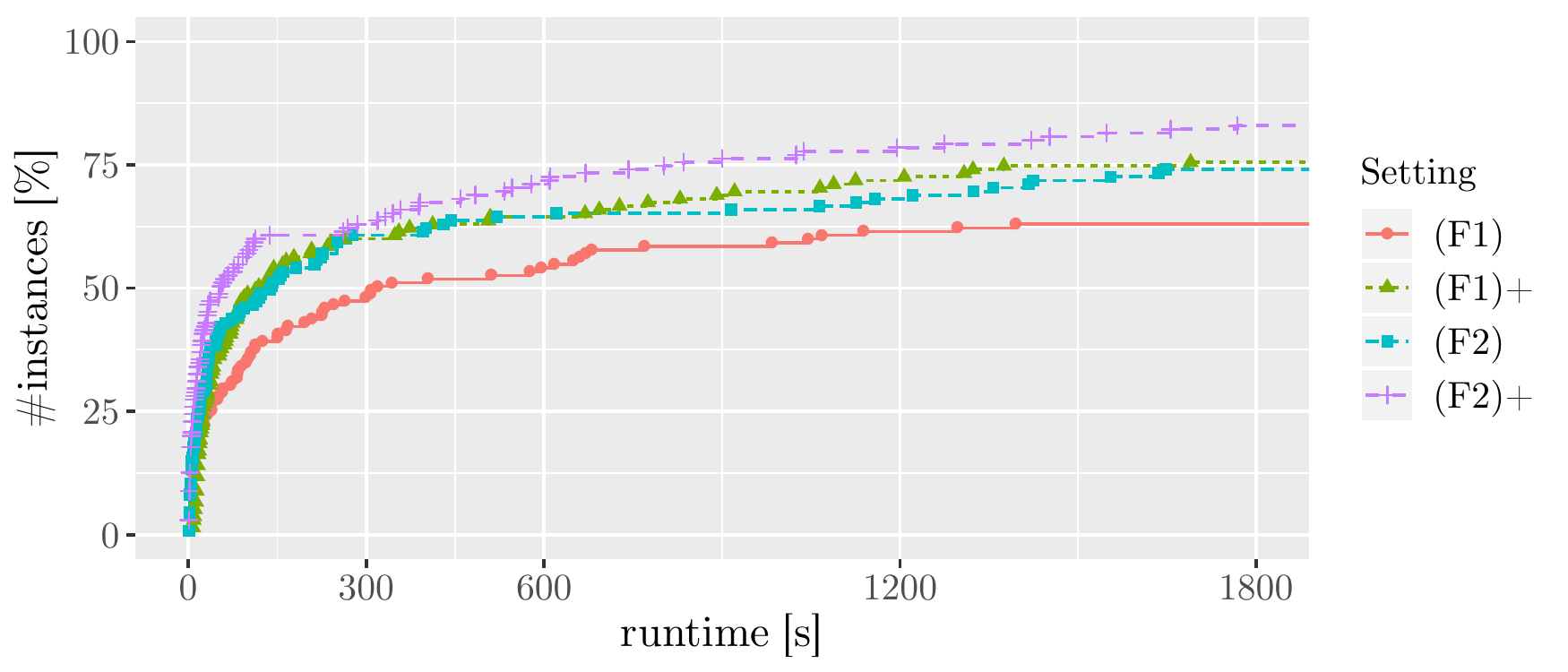}
		\caption{Runtime plot for all instances of the set}\label{fig:ilpplot1}
	\end{subfigure}%
	\begin{subfigure}[b]{.5\linewidth}
		\centering
		\centering \includegraphics[width=.99\linewidth]{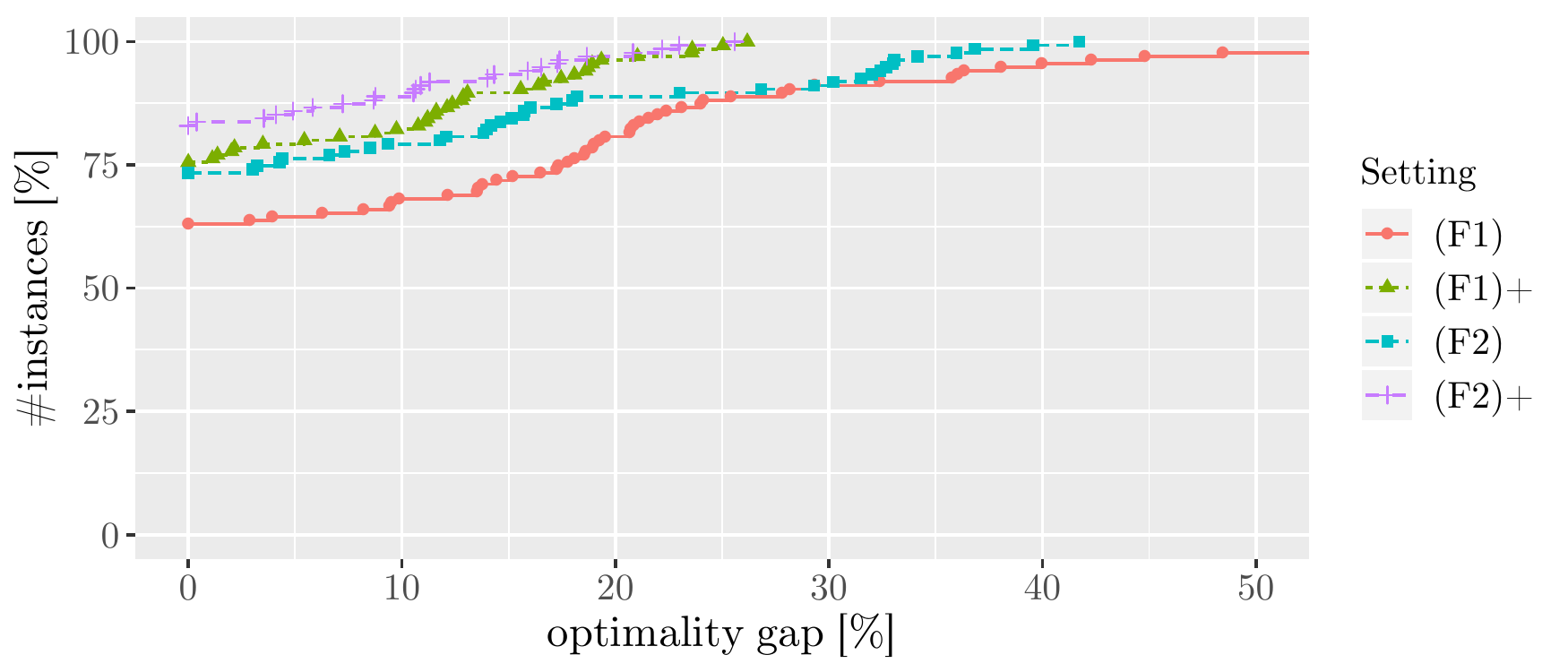}
		\caption{Optimality gap plot for all instances of the set}\label{fig:ilpplot2}
	\end{subfigure}
\newline
	\begin{subfigure}[b]{.5\linewidth}
		\centering
		\includegraphics[width=.99\linewidth]{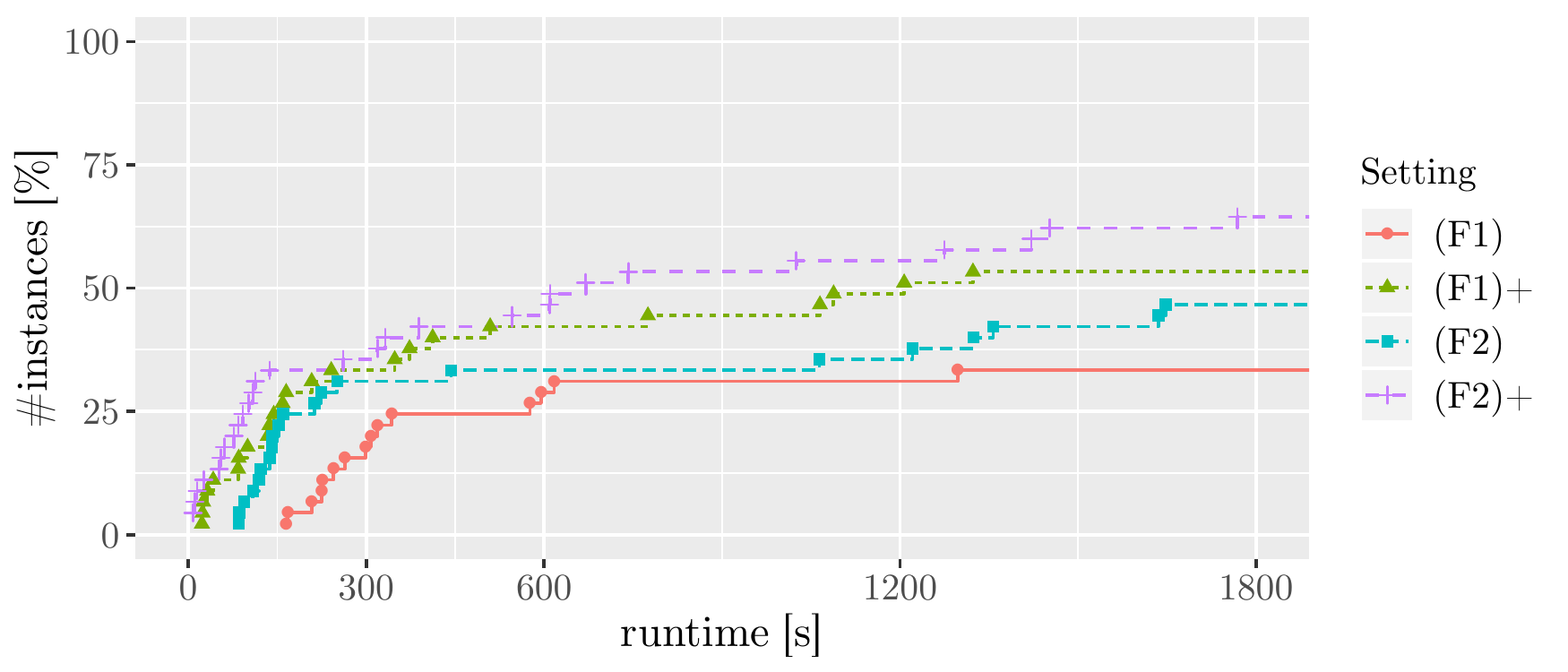}
		\caption{Runtime plot for instances with $c_u=10$}\label{fig:ilpplot3}
	\end{subfigure}%
	\begin{subfigure}[b]{.5\linewidth}
		\centering
		\includegraphics[width=.99\linewidth]{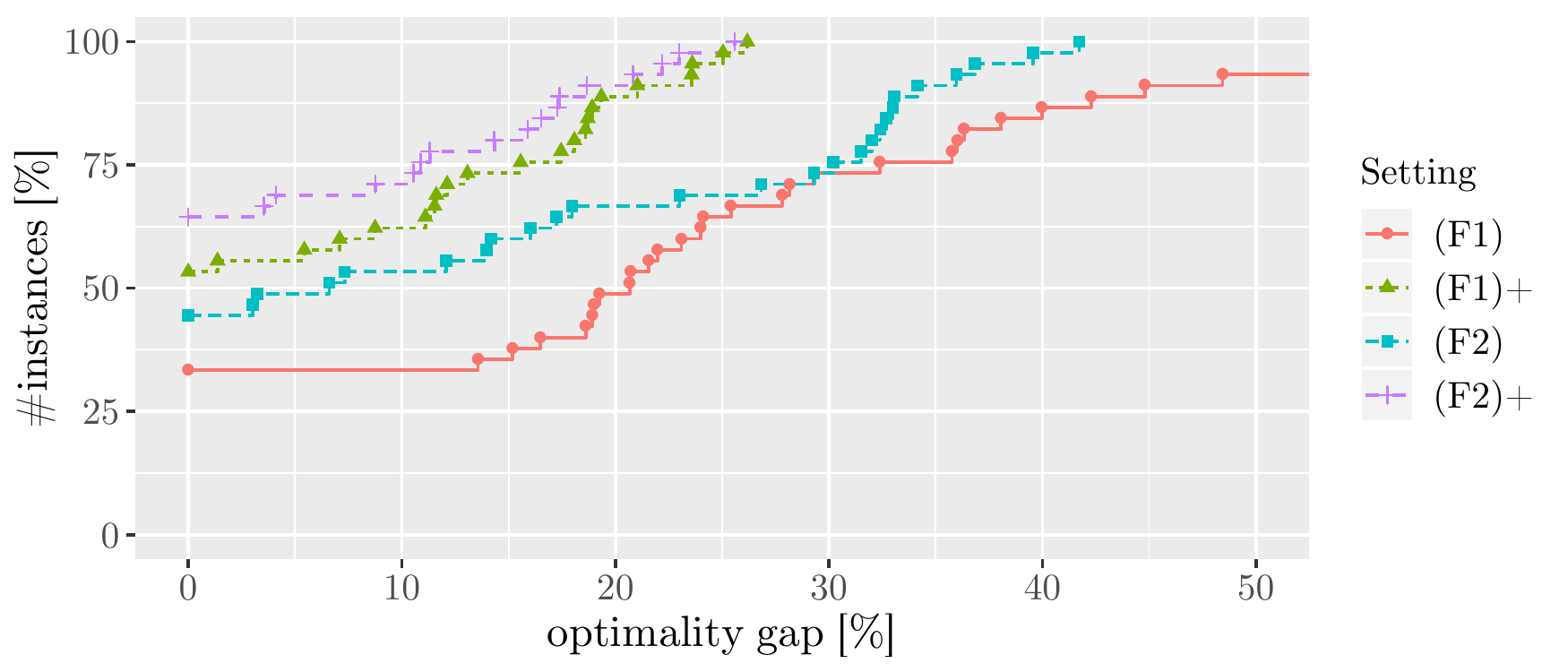}
		\caption{Optimality gap plot for instances with $c_u=10$}\label{fig:ilpplot4}
	\end{subfigure}
\newline
\begin{subfigure}[b]{.5\linewidth}
	\centering
	\includegraphics[width=.99\linewidth]{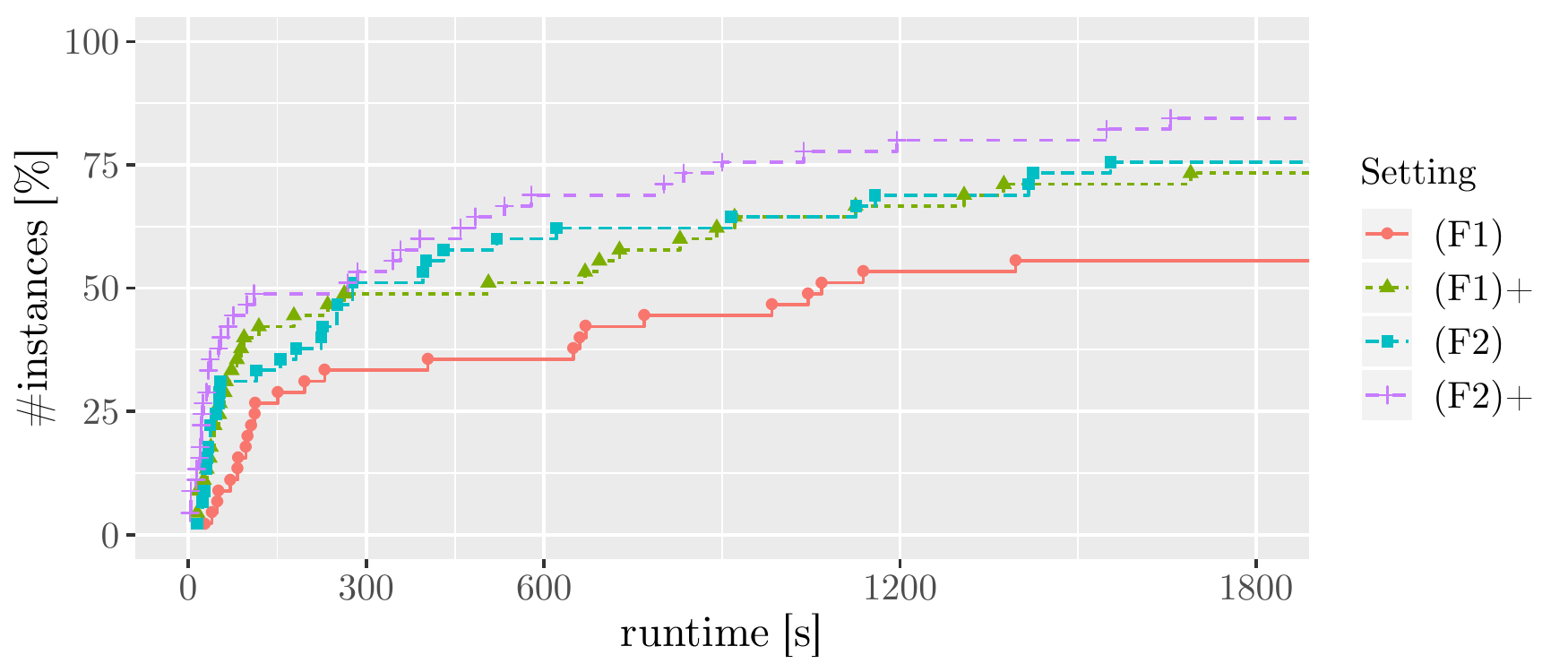}
	\caption{Runtime plot for instances with $c_u=25$}\label{fig:ilpplot5}
\end{subfigure}%
\begin{subfigure}[b]{.5\linewidth}
	\centering
	\includegraphics[width=.99\linewidth]{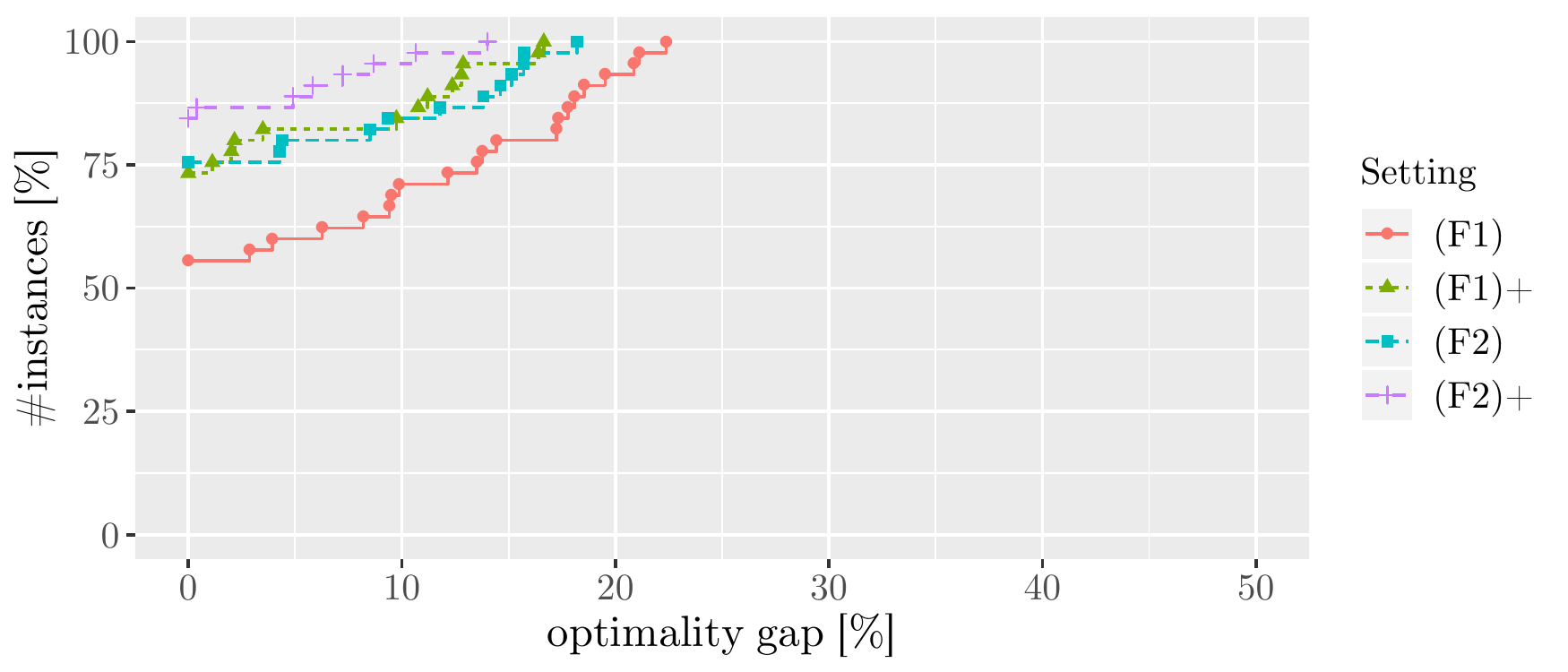}
	\caption{Optimality gap plot for instances with $c_u=25$}\label{fig:ilpplot6}
\end{subfigure}
\newline
	\centering
\begin{subfigure}[b]{.5\linewidth}
	\centering
	\includegraphics[width=.99\linewidth]{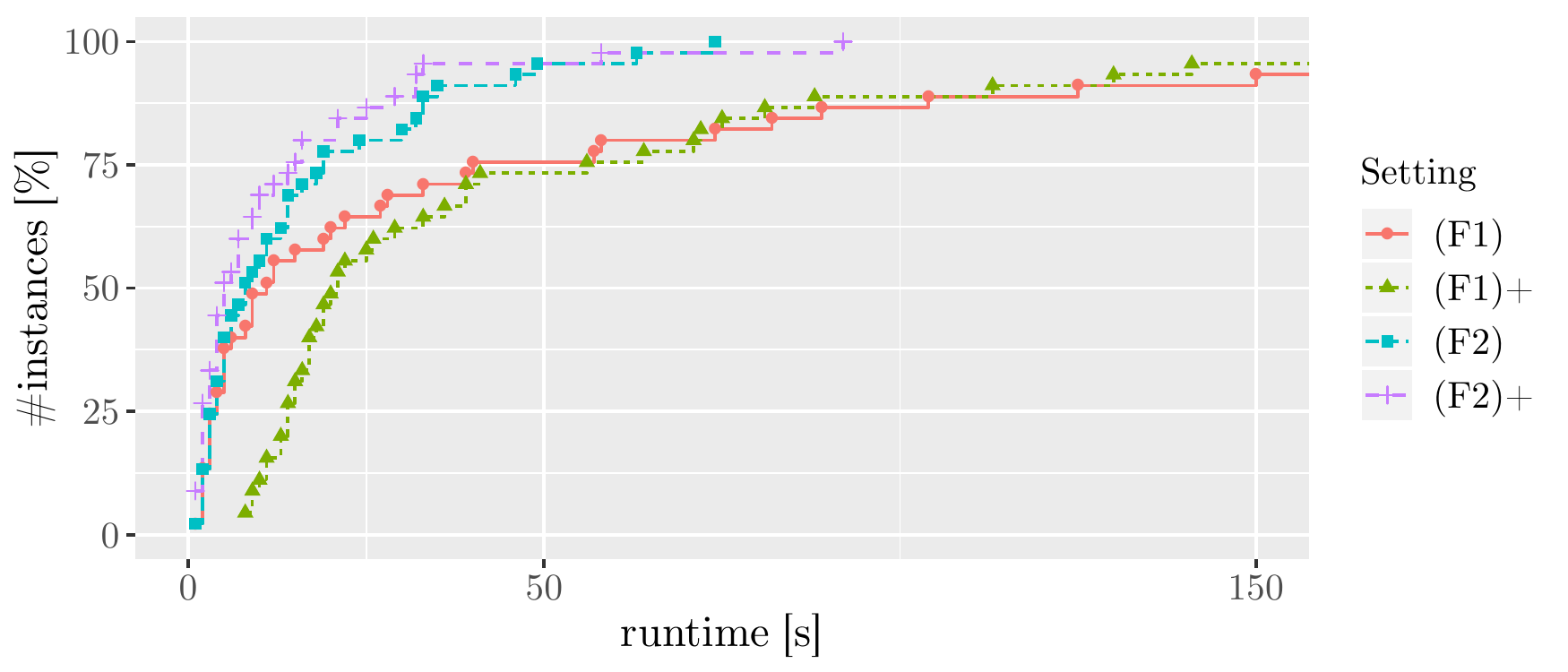}
	\caption{Runtime plot for instances with $c_u=50$ (all solved to optimality with all approaches, but x-axis cut off at 150 seconds for better readability.)}\label{fig:ilpplot7}
\end{subfigure}%

	\caption{Plots of runtimes and optimality gap of our MIP-approaches on instance set \NEWINST\ and different subgroups of these instances}\label{fig:ilpplot}
\end{figure}

In the plots shown in Figure~\ref{fig:ilpplot} we see a strong connection between the weight structure and the computational difficulty of the instances. 
All instances with $c_u=50$ can be solved by all approaches within the timelimit;
furthermore, \fsecond\ and \fsecond+ only need at most 100 seconds. 
However, for both $c_u=25$ and $c_u=10$, the situation is strikingly different, in particular, for $c_u=10$;
we can observe that only \ffirst+ and \fsecond+\ manage to solve over 50\% of the instances within the timelimit.
This behavior might be explained by the fact that, for these instances, edges play a more important role (due to their larger weight range), and thus the problem becomes more similar to a BQP problems and,
as it has been shown previously in the literature,
problems where a BQP structure plays an important role a often very hard to solve (see, e.g.,~\cite{billionnet2005different}). 
Such hypothesis could be validated by the fact that for these instances, the \ffirst+ approach, 
which includes the valid inequalities (in particular BQP-like inequalities~\ref{eq:clique}) is the approach working second best (not only when looking at the runtime, but also when looking at the optimality gap for the unsolved instances). 
Moreover, for instances with $c_u=50$ (where edge weights are less influential), 
the second best approach is \fsecond, which does not contain any valid inequalities. 
Despite these differences, when considering at all instances, we see that \fsecond+ works best, managing to solve around 80\% of the instances within the timelimit, followed by \fsecond\ and \ffirst+, which both manage to solve around 75\% of the instances.

The results reported in  Figure~\ref{fig:ilpplot} are complemented by Tables~\ref{ta:our75}-\ref{ta:our125}, 
where we give detailed results of our approaches, including the genetic algorithm (indicated by GA). 
Moreover, we also report the results obtained when running only the GRAP-part of the GA (i.e., lines~ \ref{alg:grasp1}-\ref{alg:grasp2} in Algorithm~\ref{alg:genetic}).
There is one table for each value of $|V|$ in order to allow an analysis from another point of view. 
In these tables, we present the runtime (column $t[s]$; 
TL indicates timelimit reached, and ML memorylimit), and the objective value of the best obtained solution (column $w^B$); 
for the MIP-approaches also the optimality gap (column $g[\%]$), and the number of branch-and-cut nodes (column $\#nBN$); 
and for the GA and GRASP also the primal gap compared to the best solution found by the MIP-approaches (column $pg[\%]$, calculated as $100 \cdot (w^H-w^{MIP})/w^{MIP}$, where $w^{MIP}$ is the value of the best solution found by the MIP-approaches and $w^H$ the value of the best solution found by the GRASP, resp., GA).

In the tables, we can see that all our approaches manage to solve all instances with $|V|=75$ to optimality. 
For instances with $|V|=100$, \fsecond+ solves all but two instances, and for instances with $|V|=125$, \fsecond+ solves 24 out of 45 instances to optimality within the timelimit. 
In general, for nearly all instances \fsecond+ works best, i.e., either it has the smallest runtime, or, for unsolved instances, it has the smallest optimality gap. 
For the instances, where \fsecond+ is not the best performing approach, \fsecond\ gives the best results. 
With respect to this, we can see that \fsecond\ only performs better for instances with $p=0.8$ (i.e., denser instances). 
A possible explanation for this could be, that for denser instances, the LPs with added valid inequalities becomes denser, in particular the clique inequalities, as there will also be a lot more cliques for denser graphs. 
Thus, while adding the valid inequalities improves the bound, the drawback of longer LP solution times and thus the slower node-throughput in the branch-and-cut becomes burdensome. 
This is also reflected in the number of branch-and-cut nodes enumerated, \fsecond\ often enumerates around ten times as much nodes as \fsecond+, 
while the runtime of both approaches is quite similar (and over all instances, adding the valid inequalities pays off, as only for the dense instances with $p=0.8$ the described drawback is having an effect). 
With respect to \ffirst\ and \ffirst+, the situation is similar, i.e., \ffirst\ has a considerably higher node-throughput, but in general \ffirst+ performs better. 
Moreover, when comparing \ffirst+\ and \fsecond+, it can be seen that \ffirst+\ usually needs less branch-and-cut nodes to prove optimality (when it manages to do so), but is slower than \fsecond+, as the "slimmer" formulation of \fsecond+ allows for a faster node-throughput (while still being "strong enough" for proving optimality). 
The largest optimality gap is 25.59\% and is obtained for instance \texttt{125-0.8-10-2}.

From the results reported in the tables, we can also conclude that both heuristics  perform quite well. 
The GRASP takes at most nine seconds (for some of the instances with $|V|=125$), and the largest primal gap around 20.21\% (instance \texttt{NEW-125-0.2-10-3}), while most of the primal gaps are smaller than 10\% and for slightly less than half of the instances, it is zero. The largest primal gaps are obtained for instances with $p=0.2$. 
Likewise, the GA takes at most 86 seconds (for instance \texttt{NEW-125-0.8-50-5}) and only for 30 out of 135 instances, there is a positive primal gap (the largest is 5.26\% for instance \texttt{NEW-75-0.2-10-4}, and for most instances with positive primal gap, the gap is under 1\%). 
Interestingly, for none of the unsolved instances, the GA could find an improved solution compared to the best solution found by the MIP approaches.

\begin{landscape}
\begin{table}[ht]
	\centering
	\caption{Comparison of our approaches on instance set \NEWINST\ with $|V|$=75. \label{ta:our75}} 
	\begingroup\footnotesize
	\begin{tabular}{rrr|rrrrrrrrrrrrrrrr|rrrrrr}
		\toprule
		\multicolumn{3}{c}{instance} & \multicolumn{4}{c}{\ffirst} & \multicolumn{4}{c}{\ffirst+} & \multicolumn{4}{c}{\fsecond} & \multicolumn{4}{c}{\fsecond+} & \multicolumn{3}{c}{GRASP} & \multicolumn{3}{c}{GA} \\ $p$ & $c_u$ & $id$ & $t[s]$ & $w^B$ & g[\%] &$\#nN$ &  $t[s]$ & $w^B$ & g[\%]&$\#nN$ &  $t[s]$ & $w^B$ & g[\%]& $\#nN$ & $t[s]$ & $w^B$ & g[\%]& $\#nN$ & $t[s]$ & $w^B$ & pg[\%]  & $t[s]$ & $w^B$ & pg[\%] \\ \midrule
		0.2 & 10 & 1 & 576 & 686 & 0.00 & 105831 & 32 & 686 & 0.00 & 1832 & 251 & 686 & 0.00 & 123502 & \textbf{15} & 686 & 0.00 & 1921 & 1 & 769 & 12.10 & 5 & 686 & 0.00 \\ 
		0.2 & 10 & 2 & 617 & 770 & 0.00 & 135320 & 25 & 770 & 0.00 & 1167 & 224 & 770 & 0.00 & 121349 & \textbf{8} & 770 & 0.00 & 1147 & 1 & 871 & 13.12 & 6 & 794 & 3.12 \\ 
		0.2 & 10 & 3 & 319 & 661 & 0.00 & 48274 & 23 & 661 & 0.00 & 665 & 85 & 661 & 0.00 & 46205 & \textbf{8} & 661 & 0.00 & 796 & 1 & 765 & 15.73 & 6 & 661 & 0.00 \\ 
		0.2 & 10 & 4 & 595 & 703 & 0.00 & 105611 & 42 & 703 & 0.00 & 2456 & 443 & 703 & 0.00 & 216827 & \textbf{26} & 703 & 0.00 & 2938 & 1 & 762 & 8.39 & 7 & 740 & 5.26 \\ 
		0.2 & 10 & 5 & 168 & 758 & 0.00 & 23620 & 24 & 758 & 0.00 & 1092 & 160 & 758 & 0.00 & 82778 & \textbf{11} & 758 & 0.00 & 1340 & 1 & 857 & 13.06 & 6 & 779 & 2.77 \\ 
		0.2 & 25 & 1 & 51 & 498 & 0.00 & 6617 & 16 & 498 & 0.00 & 263 & 37 & 498 & 0.00 & 14486 & \textbf{3} & 498 & 0.00 & 285 & 1 & 556 & 11.65 & 6 & 504 & 1.20 \\ 
		0.2 & 25 & 2 & 49 & 546 & 0.00 & 8616 & 16 & 546 & 0.00 & 170 & 37 & 546 & 0.00 & 18598 & \textbf{3} & 546 & 0.00 & 179 & 1 & 607 & 11.17 & 6 & 546 & 0.00 \\ 
		0.2 & 25 & 3 & 40 & 518 & 0.00 & 6505 & 15 & 518 & 0.00 & 157 & 27 & 518 & 0.00 & 10784 & \textbf{4} & 518 & 0.00 & 224 & 1 & 603 & 16.41 & 5 & 518 & 0.00 \\ 
		0.2 & 25 & 4 & 649 & 498 & 0.00 & 115335 & 36 & 498 & 0.00 & 3276 & 226 & 498 & 0.00 & 116161 & \textbf{19} & 498 & 0.00 & 3772 & 1 & 521 & 4.62 & 6 & 498 & 0.00 \\ 
		0.2 & 25 & 5 & 97 & 513 & 0.00 & 15696 & 15 & 513 & 0.00 & 245 & 54 & 513 & 0.00 & 23682 & \textbf{4} & 513 & 0.00 & 288 & 1 & 526 & 2.53 & 6 & 513 & 0.00 \\ 
		0.2 & 50 & 1 & 2 & 339 & 0.00 & 124 & 10 & 339 & 0.00 & 21 & 2 & 339 & 0.00 & 162 & \textbf{1} & 339 & 0.00 & 28 & 1 & 340 & 0.29 & 6 & 339 & 0.00 \\ 
		0.2 & 50 & 2 & 2 & 382 & 0.00 & 133 & 13 & 382 & 0.00 & 43 & 2 & 382 & 0.00 & 237 & \textbf{1} & 382 & 0.00 & 40 & 1 & 414 & 8.38 & 5 & 382 & 0.00 \\ 
		0.2 & 50 & 3 & \textbf{1} & 335 & 0.00 & 64 & 14 & 335 & 0.00 & 6 & \textbf{1} & 335 & 0.00 & 54 & \textbf{1} & 335 & 0.00 & 10 & 1 & 341 & 1.79 & 5 & 341 & 1.79 \\ 
		0.2 & 50 & 4 & 3 & 333 & 0.00 & 317 & 14 & 333 & 0.00 & 59 & 3 & 333 & 0.00 & 400 & \textbf{2} & 333 & 0.00 & 59 & 1 & 338 & 1.50 & 6 & 333 & 0.00 \\ 
		0.2 & 50 & 5 & 3 & 347 & 0.00 & 374 & 15 & 347 & 0.00 & 82 & 3 & 347 & 0.00 & 423 & \textbf{2} & 347 & 0.00 & 94 & 1 & 353 & 1.73 & 6 & 347 & 0.00 \\ 
		0.5 & 10 & 1 & 1297 & 581 & 0.00 & 99392 & 159 & 581 & 0.00 & 4820 & 213 & 581 & 0.00 & 62202 & \textbf{109} & 581 & 0.00 & 8222 & 1 & 590 & 1.55 & 13 & 581 & 0.00 \\ 
		0.5 & 10 & 2 & 299 & 602 & 0.00 & 30769 & 134 & 602 & 0.00 & 3840 & 152 & 602 & 0.00 & 41014 & \textbf{84} & 602 & 0.00 & 6719 & 1 & 641 & 6.48 & 11 & 602 & 0.00 \\ 
		0.5 & 10 & 3 & 226 & 545 & 0.00 & 19174 & 100 & 545 & 0.00 & 2739 & 141 & 545 & 0.00 & 35146 & \textbf{61} & 545 & 0.00 & 4960 & 1 & 545 & 0.00 & 10 & 545 & 0.00 \\ 
		0.5 & 10 & 4 & 264 & 540 & 0.00 & 25262 & 84 & 540 & 0.00 & 1960 & 109 & 540 & 0.00 & 28090 & \textbf{55} & 540 & 0.00 & 3797 & 1 & 580 & 7.41 & 10 & 540 & 0.00 \\ 
		0.5 & 10 & 5 & 165 & 519 & 0.00 & 13004 & 85 & 519 & 0.00 & 1803 & 119 & 519 & 0.00 & 28200 & \textbf{52} & 519 & 0.00 & 3291 & 1 & 551 & 6.17 & 10 & 519 & 0.00 \\ 
		0.5 & 25 & 1 & 106 & 387 & 0.00 & 7161 & 52 & 387 & 0.00 & 1269 & 31 & 387 & 0.00 & 7626 & \textbf{23} & 387 & 0.00 & 1598 & 1 & 402 & 3.88 & 10 & 387 & 0.00 \\ 
		0.5 & 25 & 2 & 71 & 384 & 0.00 & 5083 & 44 & 384 & 0.00 & 1194 & 24 & 384 & 0.00 & 5674 & \textbf{20} & 384 & 0.00 & 1458 & 1 & 413 & 7.55 & 10 & 384 & 0.00 \\ 
		0.5 & 25 & 3 & 83 & 362 & 0.00 & 4769 & 26 & 362 & 0.00 & 343 & 31 & 362 & 0.00 & 5898 & \textbf{13} & 362 & 0.00 & 442 & 1 & 380 & 4.97 & 10 & 362 & 0.00 \\ 
		0.5 & 25 & 4 & 100 & 366 & 0.00 & 7073 & 63 & 366 & 0.00 & 1833 & 52 & 366 & 0.00 & 14482 & \textbf{31} & 366 & 0.00 & 2551 & 1 & 371 & 1.37 & 9 & 371 & 1.37 \\ 
		0.5 & 25 & 5 & 84 & 331 & 0.00 & 4938 & 31 & 331 & 0.00 & 389 & 24 & 331 & 0.00 & 4688 & \textbf{14} & 331 & 0.00 & 470 & 1 & 331 & 0.00 & 10 & 331 & 0.00 \\ 
		0.5 & 50 & 1 & 5 & 240 & 0.00 & 348 & 16 & 240 & 0.00 & 77 & 4 & 240 & 0.00 & 424 & \textbf{3} & 240 & 0.00 & 97 & 1 & 244 & 1.67 & 9 & 240 & 0.00 \\ 
		0.5 & 50 & 2 & \textbf{2} & 238 & 0.00 & 43 & 11 & 238 & 0.00 & 28 & \textbf{2} & 238 & 0.00 & 68 & \textbf{2} & 238 & 0.00 & 30 & 1 & 245 & 2.94 & 9 & 238 & 0.00 \\ 
		0.5 & 50 & 3 & 2 & 215 & 0.00 & 0 & 8 & 215 & 0.00 & 0 & 2 & 215 & 0.00 & 47 & \textbf{1} & 215 & 0.00 & 0 & 1 & 215 & 0.00 & 9 & 215 & 0.00 \\ 
		0.5 & 50 & 4 & 8 & 235 & 0.00 & 553 & 14 & 235 & 0.00 & 141 & 5 & 235 & 0.00 & 703 & \textbf{4} & 235 & 0.00 & 134 & 1 & 235 & 0.00 & 9 & 235 & 0.00 \\ 
		0.5 & 50 & 5 & 4 & 206 & 0.00 & 198 & 11 & 206 & 0.00 & 47 & 3 & 206 & 0.00 & 198 & \textbf{2} & 206 & 0.00 & 47 & 1 & 206 & 0.00 & 8 & 206 & 0.00 \\ 
		0.8 & 10 & 1 & 343 & 571 & 0.00 & 24196 & 241 & 571 & 0.00 & 3009 & \textbf{137} & 571 & 0.00 & 23527 & \textbf{137} & 571 & 0.00 & 4798 & 2 & 613 & 7.36 & 16 & 571 & 0.00 \\ 
		0.8 & 10 & 2 & 208 & 520 & 0.00 & 12433 & 144 & 520 & 0.00 & 1481 & \textbf{86} & 520 & 0.00 & 14721 & 102 & 520 & 0.00 & 2583 & 2 & 520 & 0.00 & 15 & 520 & 0.00 \\ 
		0.8 & 10 & 3 & 245 & 543 & 0.00 & 13720 & 165 & 543 & 0.00 & 2105 & 122 & 543 & 0.00 & 19460 & \textbf{92} & 543 & 0.00 & 3475 & 2 & 543 & 0.00 & 15 & 543 & 0.00 \\ 
		0.8 & 10 & 4 & 308 & 571 & 0.00 & 17925 & 208 & 571 & 0.00 & 2722 & 142 & 571 & 0.00 & 27914 & \textbf{113} & 571 & 0.00 & 4340 & 2 & 571 & 0.00 & 15 & 571 & 0.00 \\ 
		0.8 & 10 & 5 & 225 & 509 & 0.00 & 10311 & 137 & 509 & 0.00 & 1374 & 94 & 509 & 0.00 & 15107 & \textbf{77} & 509 & 0.00 & 2412 & 2 & 509 & 0.00 & 17 & 509 & 0.00 \\ 
		0.8 & 25 & 1 & 196 & 357 & 0.00 & 6516 & 119 & 357 & 0.00 & 1527 & 53 & 357 & 0.00 & 7532 & \textbf{51} & 357 & 0.00 & 1957 & 2 & 360 & 0.84 & 15 & 357 & 0.00 \\ 
		0.8 & 25 & 2 & 151 & 338 & 0.00 & 4736 & 89 & 338 & 0.00 & 1119 & \textbf{34} & 338 & 0.00 & 4454 & \textbf{34} & 338 & 0.00 & 1201 & 2 & 356 & 5.33 & 15 & 338 & 0.00 \\ 
		0.8 & 25 & 3 & 28 & 323 & 0.00 & 1495 & 44 & 323 & 0.00 & 439 & \textbf{15} & 323 & 0.00 & 1568 & 22 & 323 & 0.00 & 549 & 2 & 323 & 0.00 & 13 & 323 & 0.00 \\ 
		0.8 & 25 & 4 & 113 & 345 & 0.00 & 4240 & 73 & 345 & 0.00 & 670 & 47 & 345 & 0.00 & 6870 & \textbf{37} & 345 & 0.00 & 947 & 2 & 345 & 0.00 & 13 & 345 & 0.00 \\ 
		0.8 & 25 & 5 & 112 & 311 & 0.00 & 3977 & 53 & 311 & 0.00 & 570 & 32 & 311 & 0.00 & 3900 & \textbf{25} & 311 & 0.00 & 747 & 2 & 311 & 0.00 & 15 & 311 & 0.00 \\ 
		0.8 & 50 & 1 & \textbf{2} & 182 & 0.00 & 29 & 8 & 182 & 0.00 & 13 & 4 & 182 & 0.00 & 136 & \textbf{2} & 182 & 0.00 & 16 & 2 & 182 & 0.00 & 14 & 182 & 0.00 \\ 
		0.8 & 50 & 2 & 3 & 188 & 0.00 & 61 & 9 & 188 & 0.00 & 33 & 3 & 188 & 0.00 & 86 & \textbf{2} & 188 & 0.00 & 30 & 2 & 188 & 0.00 & 11 & 188 & 0.00 \\ 
		0.8 & 50 & 3 & 3 & 191 & 0.00 & 64 & 9 & 191 & 0.00 & 16 & \textbf{2} & 191 & 0.00 & 35 & \textbf{2} & 191 & 0.00 & 19 & 2 & 191 & 0.00 & 11 & 191 & 0.00 \\ 
		0.8 & 50 & 4 & 5 & 196 & 0.00 & 127 & 13 & 196 & 0.00 & 67 & \textbf{4} & 196 & 0.00 & 222 & \textbf{4} & 196 & 0.00 & 62 & 2 & 196 & 0.00 & 12 & 196 & 0.00 \\ 
		0.8 & 50 & 5 & \textbf{5} & 192 & 0.00 & 176 & 15 & 192 & 0.00 & 81 & 6 & 192 & 0.00 & 402 & 7 & 192 & 0.00 & 77 & 2 & 192 & 0.00 & 15 & 192 & 0.00 \\ 
		\bottomrule
	\end{tabular}
	\endgroup
\end{table}
\begin{table}[ht]
	\centering
	\caption{Comparison of our approaches on instance set \NEWINST\ with $|V|$=100. \label{ta:our100}} 
	\begingroup\footnotesize
	\begin{tabular}{rrr|rrrrrrrrrrrrrrrr|rrrrrr}
		\toprule
		\multicolumn{3}{c}{instance} & \multicolumn{4}{c}{\ffirst} & \multicolumn{4}{c}{\ffirst+} & \multicolumn{4}{c}{\fsecond} & \multicolumn{4}{c}{\fsecond+} & \multicolumn{3}{c}{GRASP} & \multicolumn{3}{c}{GA} \\ $p$ & $c_u$ & $id$ & $t[s]$ & $w^B$ & g[\%] &$\#nN$ &  $t[s]$ & $w^B$ & g[\%]&$\#nN$ &  $t[s]$ & $w^B$ & g[\%]& $\#nN$ & $t[s]$ & $w^B$ & g[\%]& $\#nN$ & $t[s]$ & $w^B$ & pg[\%]  & $t[s]$ & $w^B$ & pg[\%] \\ \midrule
		0.2 & 10 & 1 & TL & 931 & 23.09 & 150968 & 373 & 873 &  0.00 & 17695 & TL & 914 & 16.03 & 488396 & \textbf{319} & 873 & 0.00 & 28266 & 1 & 930 & 6.53 & 12 & 873 & 0.00 \\ 
		0.2 & 10 & 2 & TL & 991 & 20.67 & 195361 & 348 & 944 &  0.00 & 16097 & TL & 966 & 13.96 & 537600 & \textbf{261} & 944 & 0.00 & 21216 & 1 & 983 & 4.13 & 13 & 944 & 0.00 \\ 
		0.2 & 10 & 3 & TL & 933 & 21.97 & 175869 & 509 & 878 &  0.00 & 24488 & TL & 937 & 17.98 & 488300 & \textbf{389} & 878 & 0.00 & 32623 & 1 & 905 & 3.08 & 11 & 878 & 0.00 \\ 
		0.2 & 10 & 4 & TL & 837 & 19.00 & 160900 & 775 & 837 &  0.00 & 36727 & TL & 850 & 14.18 & 533800 & \textbf{546} & 837 & 0.00 & 44398 & 1 & 879 & 5.02 & 11 & 837 & 0.00 \\ 
		0.2 & 10 & 5 & TL & 913 & 23.99 & 166974 & 412 & 840 &  0.00 & 19910 & TL & 847 & 12.08 & 455300 & \textbf{332} & 840 & 0.00 & 31130 & 1 & 907 & 7.98 & 12 & 870 & 3.57 \\ 
		0.2 & 25 & 1 & 660 & 591 &  0.00 & 52158 & 82 & 591 &  0.00 & 3688 & 401 & 591 &  0.00 & 97432 & \textbf{55} & 591 & 0.00 & 5421 & 1 & 591 & 0.00 & 12 & 591 & 0.00 \\ 
		0.2 & 25 & 2 & TL & 655 &  3.93 & 154809 & 94 & 653 &  0.00 & 4420 & 1126 & 653 &  0.00 & 279737 & \textbf{67} & 653 & 0.00 & 7342 & 1 & 687 & 5.21 & 11 & 655 & 0.31 \\ 
		0.2 & 25 & 3 & 769 & 612 &  0.00 & 64664 & 61 & 612 &  0.00 & 2355 & 251 & 612 &  0.00 & 56158 & \textbf{34} & 612 & 0.00 & 3142 & 1 & 648 & 5.88 & 12 & 616 & 0.65 \\ 
		0.2 & 25 & 4 & TL & 558 &  2.87 & 122658 & 38 & 552 &  0.00 & 917 & 224 & 552 &  0.00 & 53608 & \textbf{22} & 552 & 0.00 & 1716 & 1 & 602 & 9.06 & 11 & 552 & 0.00 \\ 
		0.2 & 25 & 5 & TL & 606 &  6.27 & 172508 & 506 & 606 &  0.00 & 31561 & TL & 609 &  4.40 & 401219 & \textbf{345} & 606 & 0.00 & 40428 & 1 & 646 & 6.60 & 12 & 607 & 0.17 \\ 
		0.2 & 50 & 1 & 11 & 418 &  0.00 & 929 & 19 & 418 &  0.00 & 193 & 8 & 418 &  0.00 & 1247 & \textbf{5} & 418 & 0.00 & 249 & 1 & 422 & 0.96 & 12 & 420 & 0.48 \\ 
		0.2 & 50 & 2 & 9 & 447 &  0.00 & 774 & 18 & 447 &  0.00 & 177 & 8 & 447 &  0.00 & 1154 & \textbf{7} & 447 & 0.00 & 261 & 1 & 472 & 5.59 & 11 & 456 & 2.01 \\ 
		0.2 & 50 & 3 & 5 & 419 &  0.00 & 339 & 17 & 419 &  0.00 & 106 & 6 & 419 &  0.00 & 590 & \textbf{4} & 419 & 0.00 & 124 & 1 & 427 & 1.91 & 11 & 419 & 0.00 \\ 
		0.2 & 50 & 4 & 74 & 403 &  0.00 & 4679 & 21 & 403 &  0.00 & 329 & 19 & 403 &  0.00 & 3709 & \textbf{10} & 403 & 0.00 & 395 & 1 & 418 & 3.72 & 12 & 410 & 1.74 \\ 
		0.2 & 50 & 5 & 12 & 375 &  0.00 & 800 & 20 & 375 &  0.00 & 290 & 10 & 375 &  0.00 & 1663 & \textbf{5} & 375 & 0.00 & 328 & 1 & 379 & 1.07 & 13 & 379 & 1.07 \\ 
		0.5 & 10 & 1 & TL & 763 & 18.61 & 104422 & TL & 743 &  7.09 & 26400 & TL & 743 &  7.32 & 240277 & \textbf{1421} & 743 & 0.00 & 61206 & 2 & 749 & 0.81 & 26 & 749 & 0.81 \\ 
		0.5 & 10 & 2 & TL & 708 & 28.16 & 112961 & 1207 & 698 &  0.00 & 15696 & 1357 & 698 &  0.00 & 226826 & \textbf{670} & 698 & 0.00 & 25318 & 3 & 705 & 1.00 & 25 & 700 & 0.29 \\ 
		0.5 & 10 & 3 & TL & 728 & 16.48 & 103135 & 1323 & 699 &  0.00 & 19211 & TL & 699 &  3.00 & 291119 & \textbf{742} & 699 & 0.00 & 29592 & 3 & 730 & 4.43 & 24 & 718 & 2.72 \\ 
		0.5 & 10 & 4 & TL & 726 & 13.57 & 107121 & 1088 & 726 &  0.00 & 13761 & 1324 & 726 &  0.00 & 218790 & \textbf{609} & 726 & 0.00 & 22324 & 2 & 775 & 6.75 & 26 & 726 & 0.00 \\ 
		0.5 & 10 & 5 & TL & 761 & 24.11 & 124466 & TL & 702 &  1.37 & 24691 & TL & 744 & 17.25 & 240100 & \textbf{1275} & 702 & 0.00 & 51404 & 2 & 743 & 5.84 & 25 & 702 & 0.00 \\ 
		0.5 & 25 & 1 & 670 & 461 &  0.00 & 18182 & 235 & 461 &  0.00 & 2640 & 182 & 461 &  0.00 & 22792 & \textbf{99} & 461 & 0.00 & 3913 & 3 & 461 & 0.00 & 25 & 461 & 0.00 \\ 
		0.5 & 25 & 2 & 230 & 437 &  0.00 & 6776 & 178 & 437 &  0.00 & 2454 & 115 & 437 &  0.00 & 12140 & \textbf{76} & 437 & 0.00 & 3372 & 2 & 448 & 2.52 & 19 & 437 & 0.00 \\ 
		0.5 & 25 & 3 & 404 & 434 &  0.00 & 10685 & 263 & 434 &  0.00 & 3821 & 155 & 434 &  0.00 & 16969 & \textbf{111} & 434 & 0.00 & 4425 & 3 & 443 & 2.07 & 22 & 434 & 0.00 \\ 
		0.5 & 25 & 4 & TL & 494 &  8.20 & 63896 & 921 & 482 &  0.00 & 16949 & 621 & 482 &  0.00 & 90411 & \textbf{533} & 482 & 0.00 & 22037 & 2 & 489 & 1.45 & 25 & 482 & 0.00 \\ 
		0.5 & 25 & 5 & 1395 & 456 &  0.00 & 40173 & 829 & 456 &  0.00 & 12615 & 430 & 456 &  0.00 & 59506 & \textbf{358} & 456 & 0.00 & 16885 & 3 & 470 & 3.07 & 23 & 457 & 0.22 \\ 
		0.5 & 50 & 1 & 4 & 260 &  0.00 & 27 & 17 & 260 &  0.00 & 5 & 5 & 260 &  0.00 & 131 & \textbf{3} & 260 & 0.00 & 5 & 2 & 260 & 0.00 & 22 & 260 & 0.00 \\ 
		0.5 & 50 & 2 & 3 & 271 &  0.00 & 27 & 17 & 271 &  0.00 & 15 & 3 & 271 &  0.00 & 55 & \textbf{2} & 271 & 0.00 & 14 & 2 & 271 & 0.00 & 21 & 271 & 0.00 \\ 
		0.5 & 50 & 3 & 9 & 283 &  0.00 & 282 & 21 & 283 &  0.00 & 119 & 7 & 283 &  0.00 & 404 & \textbf{4} & 283 & 0.00 & 135 & 3 & 283 & 0.00 & 21 & 283 & 0.00 \\ 
		0.5 & 50 & 4 & 27 & 291 &  0.00 & 914 & 39 & 291 &  0.00 & 353 & 11 & 291 &  0.00 & 1070 & \textbf{10} & 291 & 0.00 & 355 & 2 & 296 & 1.72 & 22 & 291 & 0.00 \\ 
		0.5 & 50 & 5 & 12 & 269 &  0.00 & 347 & 29 & 269 &  0.00 & 228 & 14 & 269 &  0.00 & 1254 & \textbf{9} & 269 & 0.00 & 251 & 2 & 269 & 0.00 & 21 & 269 & 0.00 \\ 
		0.8 & 10 & 1 & TL & 730 & 20.71 & 55600 & TL & 730 & 15.57 & 11878 & TL & 730 & \textbf{ 3.22} & 176962 & TL & 730 & 8.77 & 30496 & 4 & 730 & 0.00 & 39 & 730 & 0.00 \\ 
		0.8 & 10 & 2 & TL & 697 & 15.18 & 40114 & TL & 683 & 11.61 & 5773 & 1064 & 683 &  0.00 & 103180 & \textbf{1025} & 683 & 0.00 & 17483 & 4 & 688 & 0.73 & 37 & 683 & 0.00 \\ 
		0.8 & 10 & 3 & TL & 721 & 19.24 & 47870 & TL & 718 & 11.53 & 10506 & \textbf{1636} & 718 &  0.00 & 154346 & 1769 & 718 & 0.00 & 31269 & 4 & 718 & 0.00 & 37 & 718 & 0.00 \\ 
		0.8 & 10 & 4 & TL & 726 & 36.02 & 52165 & TL & 709 &  8.75 & 9566 & TL & 712 &  6.60 & 153824 & \textbf{1452} & 709 & 0.00 & 28487 & 4 & 709 & 0.00 & 41 & 709 & 0.00 \\ 
		0.8 & 10 & 5 & TL & 703 & 18.92 & 43898 & TL & 700 & 18.08 & 7205 & \textbf{1221} & 700 &  0.00 & 118429 & TL & 700 & 3.55 & 28770 & 4 & 710 & 1.43 & 39 & 704 & 0.57 \\ 
		0.8 & 25 & 1 & 1138 & 442 &  0.00 & 15789 & 1125 & 442 &  0.00 & 7017 & \textbf{396} & 442 &  0.00 & 34791 & 459 & 442 & 0.00 & 9633 & 5 & 452 & 2.26 & 40 & 442 & 0.00 \\ 
		0.8 & 25 & 2 & 1068 & 430 &  0.00 & 15155 & 693 & 430 &  0.00 & 4218 & \textbf{277} & 430 &  0.00 & 27907 & 285 & 430 & 0.00 & 5284 & 4 & 430 & 0.00 & 32 & 430 & 0.00 \\ 
		0.8 & 25 & 3 & 984 & 426 &  0.00 & 15999 & 669 & 426 &  0.00 & 4180 & \textbf{251} & 426 &  0.00 & 19709 & 269 & 426 & 0.00 & 5152 & 4 & 426 & 0.00 & 36 & 426 & 0.00 \\ 
		0.8 & 25 & 4 & 1045 & 428 &  0.00 & 17287 & 891 & 428 &  0.00 & 5285 & \textbf{277} & 428 &  0.00 & 27607 & 390 & 428 & 0.00 & 7049 & 4 & 428 & 0.00 & 35 & 428 & 0.00 \\ 
		0.8 & 25 & 5 & TL & 447 &  9.51 & 26364 & 1375 & 432 &  0.00 & 8047 & \textbf{520} & 432 &  0.00 & 51363 & 578 & 432 & 0.00 & 10357 & 4 & 432 & 0.00 & 42 & 432 & 0.00 \\ 
		0.8 & 50 & 1 & 33 & 259 &  0.00 & 993 & 75 & 259 &  0.00 & 396 & \textbf{16} & 259 &  0.00 & 1415 & 21 & 259 & 0.00 & 427 & 4 & 259 & 0.00 & 32 & 259 & 0.00 \\ 
		0.8 & 50 & 2 & 6 & 246 &  0.00 & 41 & 25 & 246 &  0.00 & 33 & \textbf{5} & 246 &  0.00 & 96 & 6 & 246 & 0.00 & 44 & 4 & 246 & 0.00 & 9 & 246 & 0.00 \\ 
		0.8 & 50 & 3 & 9 & 238 &  0.00 & 106 & 26 & 238 &  0.00 & 31 & \textbf{5} & 238 &  0.00 & 154 & \textbf{5} & 238 & 0.00 & 39 & 4 & 238 & 0.00 & 34 & 238 & 0.00 \\ 
		0.8 & 50 & 4 & 28 & 253 &  0.00 & 673 & 56 & 253 &  0.00 & 210 & \textbf{14} & 253 &  0.00 & 757 & 16 & 253 & 0.00 & 232 & 4 & 258 & 1.98 & 34 & 253 & 0.00 \\ 
		0.8 & 50 & 5 & 39 & 248 &  0.00 & 1042 & 81 & 248 &  0.00 & 414 & \textbf{18} & 248 &  0.00 & 1428 & 25 & 248 & 0.00 & 451 & 5 & 250 & 0.81 & 31 & 248 & 0.00 \\ 
		\bottomrule
	\end{tabular}
	\endgroup
\end{table}
\begin{table}[ht]
	\centering
	\caption{Comparison of our approaches on instance set \NEWINST\ with $|V|$=125. \label{ta:our125}} 
	\begingroup\footnotesize
	\begin{tabular}{rrr|rrrrrrrrrrrrrrrr|rrrrrr}
		\toprule
		\multicolumn{3}{c}{instance} & \multicolumn{4}{c}{\ffirst} & \multicolumn{4}{c}{\ffirst+} & \multicolumn{4}{c}{\fsecond} & \multicolumn{4}{c}{\fsecond+} & \multicolumn{3}{c}{GRASP} & \multicolumn{3}{c}{GA} \\ $p$ & $c_u$ & $id$ & $t[s]$ & $w^B$ & g[\%] &$\#nN$ &  $t[s]$ & $w^B$ & g[\%]&$\#nN$ &  $t[s]$ & $w^B$ & g[\%]& $\#nN$ & $t[s]$ & $w^B$ & g[\%]& $\#nN$ & $t[s]$ & $w^B$ & pg[\%]  & $t[s]$ & $w^B$ & pg[\%] \\ \midrule
		0.2 & 10 & 1 & TL & 1031 & 29.33 & 111233 & TL & 1031 & 13.08 & 38900 & TL & 1122 & 33.07 & 289800 & TL & 1026 & \textbf{11.31} & 66500 & 2 & 1112 & 8.38 & 24 & 1026 & 0.00 \\ 
		0.2 & 10 & 2 & TL & 1038 & 25.41 & 103199 & TL & 1038 &  5.43 & 39364 & TL & 1136 & 29.30 & 323400 & TL & 1038 & \textbf{ 4.11} & 70700 & 2 & 1069 & 2.99 & 22 & 1038 & 0.00 \\ 
		0.2 & 10 & 3 & TL & 935 & 21.55 & 112721 & 1065 & 935 &  0.00 & 18545 & TL & 1006 & 23.01 & 307900 & \textbf{610} & 935 &  0.00 & 23794 & 2 & 1124 & 20.21 & 23 & 947 & 1.28 \\ 
		0.2 & 10 & 4 & TL & 1087 & 32.38 & 162826 & TL & 1050 & 11.10 & 35800 & TL & 1102 & 30.22 & 307500 & TL & 1052 & \textbf{10.55} & 65400 & 2 & 1121 & 6.76 & 21 & 1051 & 0.10 \\ 
		0.2 & 10 & 5 & TL & 1067 & 38.06 & 98022 & TL & 978 & 12.12 & 46300 & TL & 1069 & 32.41 & 293644 & TL & 974 & \textbf{10.88} & 72100 & 2 & 1112 & 14.17 & 25 & 975 & 0.10 \\ 
		0.2 & 25 & 1 & TL & 752 & 14.43 & 79324 & 727 & 720 &  0.00 & 20101 & TL & 777 & 15.13 & 249000 & \textbf{484} & 720 &  0.00 & 30681 & 2 & 803 & 11.53 & 26 & 720 & 0.00 \\ 
		0.2 & 25 & 2 & TL & 748 &  9.42 & 115536 & 1690 & 746 &  0.00 & 49679 & TL & 755 &  9.34 & 250400 & \textbf{1038} & 746 &  0.00 & 66326 & 2 & 768 & 2.95 & 24 & 748 & 0.27 \\ 
		0.2 & 25 & 3 & TL & 758 & 17.24 & 76593 & 1308 & 715 &  0.00 & 32387 & TL & 756 & 13.82 & 262200 & \textbf{802} & 715 &  0.00 & 58391 & 2 & 752 & 5.17 & 21 & 717 & 0.28 \\ 
		0.2 & 25 & 4 & TL & 725 & 13.52 & 125557 & TL & 701 &  1.13 & 45548 & TL & 726 & 11.79 & 277800 & \textbf{1195} & 701 &  0.00 & 68666 & 2 & 726 & 3.57 & 22 & 705 & 0.57 \\ 
		0.2 & 25 & 5 & TL & 690 & 12.15 & 125451 & TL & 684 &  3.49 & 48278 & TL & 714 & 14.62 & 284000 & \textbf{1548} & 684 &  0.00 & 94996 & 2 & 747 & 9.21 & 23 & 697 & 1.90 \\ 
		0.2 & 50 & 1 & 22 & 455 &  0.00 & 914 & 19 & 455 &  0.00 & 94 & 14 & 455 &  0.00 & 1809 & \textbf{3} & 455 &  0.00 & 112 & 2 & 457 & 0.44 & 21 & 455 & 0.00 \\ 
		0.2 & 50 & 2 & 15 & 477 &  0.00 & 552 & 22 & 477 &  0.00 & 163 & 11 & 477 &  0.00 & 1216 & \textbf{4} & 477 &  0.00 & 153 & 2 & 493 & 3.35 & 23 & 477 & 0.00 \\ 
		0.2 & 50 & 3 & 150 & 490 &  0.00 & 5438 & 33 & 490 &  0.00 & 379 & 32 & 490 &  0.00 & 4963 & \textbf{9} & 490 &  0.00 & 446 & 2 & 501 & 2.24 & 21 & 490 & 0.00 \\ 
		0.2 & 50 & 4 & 307 & 467 &  0.00 & 10476 & 36 & 467 &  0.00 & 678 & 63 & 467 &  0.00 & 11763 & \textbf{14} & 467 &  0.00 & 903 & 2 & 504 & 7.92 & 23 & 467 & 0.00 \\ 
		0.2 & 50 & 5 & 680 & 457 &  0.00 & 27859 & 71 & 457 &  0.00 & 1974 & 74 & 457 &  0.00 & 12890 & \textbf{29} & 457 &  0.00 & 2719 & 2 & 468 & 2.41 & 24 & 459 & 0.44 \\ 
		0.5 & 10 & 1 & TL & 888 & 35.77 & 74241 & TL & 817 & 19.35 & 11969 & TL & 920 & 32.99 & 189400 & TL & 817 & \textbf{15.90} & 32310 & 4 & 817 & 0.00 & 41 & 817 & 0.00 \\ 
		0.5 & 10 & 2 & TL & 838 & 27.80 & 71722 & TL & 815 & 18.60 & 11600 & TL & 902 & 35.97 & 165500 & TL & 815 & \textbf{14.33} & 28242 & 5 & 827 & 1.47 & 45 & 815 & 0.00 \\ 
		0.5 & 10 & 3 & TL & 931 & 48.44 & 71111 & TL & 836 & 21.04 & 12000 & TL & 915 & 32.01 & 183200 & TL & 836 & \textbf{18.68} & 31000 & 4 & 880 & 5.26 & 45 & 872 & 4.31 \\ 
		0.5 & 10 & 4 & TL & 912 & 36.32 & 81756 & TL & 867 & 23.60 & 12100 & ML & 947 & 34.16 & 194451 & TL & 867 & \textbf{20.84} & 28328 & 4 & 914 & 5.42 & 55 & 867 & 0.00 \\ 
		0.5 & 10 & 5 & TL & 949 & 39.97 & 76935 & TL & 867 & 25.04 & 12998 & TL & 995 & 41.73 & 188520 & TL & 867 & \textbf{22.20} & 30407 & 5 & 906 & 4.50 & 55 & 867 & 0.00 \\ 
		0.5 & 25 & 1 & TL & 613 & 21.13 & 37273 & TL & 566 &  9.75 & 13891 & TL & 566 & \textbf{ 4.27} & 160389 & TL & 566 &  4.91 & 37868 & 5 & 566 & 0.00 & 48 & 566 & 0.00 \\ 
		0.5 & 25 & 2 & TL & 542 &  9.87 & 30162 & TL & 533 &  2.02 & 14217 & 915 & 533 &  0.00 & 78306 & \textbf{900} & 533 &  0.00 & 24650 & 5 & 561 & 5.25 & 48 & 533 & 0.00 \\ 
		0.5 & 25 & 3 & TL & 563 & 13.77 & 30148 & TL & 538 &  2.16 & 12186 & 1417 & 538 &  0.00 & 111362 & \textbf{835} & 538 &  0.00 & 19259 & 5 & 567 & 5.39 & 49 & 538 & 0.00 \\ 
		0.5 & 25 & 4 & TL & 567 & 18.54 & 37033 & TL & 552 & 16.40 & 10610 & TL & 576 & 15.71 & 149600 & TL & 552 & \textbf{10.66} & 37400 & 4 & 565 & 2.36 & 53 & 552 & 0.00 \\ 
		0.5 & 25 & 5 & TL & 572 & 19.52 & 38695 & TL & 545 & 12.36 & 15193 & TL & 552 & \textbf{ 8.51} & 148100 & TL & 545 &  8.67 & 40091 & 5 & 548 & 0.55 & 48 & 548 & 0.55 \\ 
		0.5 & 50 & 1 & 40 & 334 &  0.00 & 785 & 64 & 334 &  0.00 & 473 & 19 & 334 &  0.00 & 1495 & \textbf{16} & 334 &  0.00 & 481 & 4 & 336 & 0.60 & 40 & 334 & 0.00 \\ 
		0.5 & 50 & 2 & 19 & 330 &  0.00 & 500 & 41 & 330 &  0.00 & 251 & 13 & 330 &  0.00 & 631 & \textbf{12} & 330 &  0.00 & 255 & 4 & 330 & 0.00 & 38 & 330 & 0.00 \\ 
		0.5 & 50 & 3 & 20 & 315 &  0.00 & 247 & 39 & 315 &  0.00 & 80 & 9 & 315 &  0.00 & 219 & \textbf{7} & 315 &  0.00 & 77 & 5 & 315 & 0.00 & 49 & 315 & 0.00 \\ 
		0.5 & 50 & 4 & 57 & 316 &  0.00 & 834 & 88 & 316 &  0.00 & 488 & 33 & 316 &  0.00 & 2657 & \textbf{21} & 316 &  0.00 & 504 & 5 & 316 & 0.00 & 51 & 316 & 0.00 \\ 
		0.5 & 50 & 5 & 104 & 311 &  0.00 & 2479 & 113 & 311 &  0.00 & 1099 & 33 & 311 &  0.00 & 3111 & \textbf{32} & 311 &  0.00 & 1107 & 4 & 311 & 0.00 & 40 & 311 & 0.00 \\ 
		0.8 & 10 & 1 & TL & 855 & 53.04 & 33869 & TL & 793 & 18.91 & 4892 & TL & 855 & 32.69 & 123500 & TL & 793 & \textbf{17.38} & 15086 & 9 & 793 & 0.00 & 78 & 793 & 0.00 \\ 
		0.8 & 10 & 2 & TL & 913 & 44.80 & 35253 & TL & 853 & 26.18 & 6445 & TL & 899 & 36.83 & 117000 & TL & 845 & \textbf{25.59} & 17975 & 8 & 854 & 1.07 & 72 & 845 & 0.00 \\ 
		0.8 & 10 & 3 & TL & 885 & 42.29 & 34230 & TL & 787 & 18.71 & 4916 & TL & 841 & 26.83 & 107600 & TL & 787 & \textbf{17.29} & 16300 & 9 & 829 & 5.34 & 74 & 787 & 0.00 \\ 
		0.8 & 10 & 4 & TL & 853 & 55.10 & 34257 & TL & 777 & 17.46 & 4700 & TL & 830 & 31.51 & 109100 & TL & 777 & \textbf{16.52} & 15100 & 9 & 829 & 6.69 & 83 & 777 & 0.00 \\ 
		0.8 & 10 & 5 & TL & 865 & 58.98 & 34289 & TL & 820 & 23.57 & 5188 & TL & 904 & 39.57 & 114502 & TL & 813 & \textbf{23.00} & 16200 & 8 & 827 & 1.72 & 77 & 813 & 0.00 \\ 
		0.8 & 25 & 1 & TL & 514 & 18.09 & 18822 & TL & 508 & 12.79 & 6100 & \textbf{1555} & 508 &  0.00 & 100191 & TL & 508 &  7.23 & 16220 & 9 & 521 & 2.56 & 69 & 510 & 0.39 \\ 
		0.8 & 25 & 2 & TL & 504 & 17.32 & 13406 & TL & 498 & 10.76 & 6000 & \textbf{1158} & 498 &  0.00 & 75456 & 1656 & 498 &  0.00 & 16200 & 9 & 499 & 0.20 & 65 & 498 & 0.00 \\ 
		0.8 & 25 & 3 & TL & 533 & 20.87 & 17604 & TL & 513 & 12.87 & 5558 & TL & 550 & 15.73 & 107800 & TL & 513 & \textbf{ 5.83} & 16743 & 9 & 523 & 1.95 & 77 & 513 & 0.00 \\ 
		0.8 & 25 & 4 & TL & 505 & 17.77 & 19977 & TL & 493 & 11.21 & 5241 & \textbf{1424} & 493 &  0.00 & 92315 & TL & 493 &  0.39 & 17238 & 8 & 506 & 2.64 & 75 & 493 & 0.00 \\ 
		0.8 & 25 & 5 & TL & 515 & 22.38 & 16373 & TL & 504 & 16.65 & 7000 & TL & 528 & 18.21 & 114500 & TL & 504 & \textbf{14.02} & 19469 & 8 & 519 & 2.98 & 76 & 504 & 0.00 \\ 
		0.8 & 50 & 1 & 511 & 307 &  0.00 & 4416 & 355 & 307 &  0.00 & 1499 & \textbf{49} & 307 &  0.00 & 2868 & 92 & 307 &  0.00 & 1568 & 8 & 307 & 0.00 & 64 & 307 & 0.00 \\ 
		0.8 & 50 & 2 & 58 & 296 &  0.00 & 897 & 130 & 296 &  0.00 & 360 & \textbf{24} & 296 &  0.00 & 1484 & 32 & 296 &  0.00 & 397 & 8 & 296 & 0.00 & 57 & 296 & 0.00 \\ 
		0.8 & 50 & 3 & 125 & 294 &  0.00 & 1888 & 141 & 294 &  0.00 & 316 & \textbf{30} & 294 &  0.00 & 1655 & 33 & 294 &  0.00 & 351 & 8 & 294 & 0.00 & 71 & 294 & 0.00 \\ 
		0.8 & 50 & 4 & 82 & 270 &  0.00 & 974 & 72 & 270 &  0.00 & 105 & 35 & 270 &  0.00 & 1747 & \textbf{15} & 270 &  0.00 & 116 & 9 & 270 & 0.00 & 86 & 270 & 0.00 \\ 
		0.8 & 50 & 5 & 89 & 278 &  0.00 & 1326 & 206 & 278 &  0.00 & 659 & \textbf{46} & 278 &  0.00 & 2611 & 58 & 278 &  0.00 & 744 & 9 & 278 & 0.00 & 77 & 278 & 0.00 \\ 
		\bottomrule
	\end{tabular}
	\endgroup
\end{table}

\end{landscape}
\clearpage

\section{Conclusions and future work \label{sec:con}}

In this paper, we presented exact and heuristic solution algorithms for the recently introduced \emph{(minimum) weighted total domination problem (\WTDP)} (see~\cite{MaEtAl2019}). 
The \WTDP\ is a problem from the family of domination problems, which are among the most basic combinatinatorial problems in graph optimization. 
In the \WTDP\ we are not just concerned with the concept of domination (i.e., finding a vertex-set $D\subset V$ for a given graph $G=(V,E)$, such that each vertex is either in $D$ or adjacent to it), 
but with the stronger concept of \emph{total domination}, which imposes that for each vertex $v \in D$, there is also a neighbor of $v$ in $D$ (i.e., the vertices of $D$ also need to be dominated by $D$). 
In the \WTDP, we have a weight function associated with the vertices and edges of the graph. 
The goal is to find a total dominating set $D$ with minimal weight. 
The weight counted in the objective is the weight of the vertices selected for $D$, the weight of the edges between vertices in $D$, and for each vertex in $V\setminus D$, 
the smallest weight of an edge between it and a vertex in $D$. 

We introduced two new Mixed-Integer Programming models for the problem, and designed solution frameworks based on them. These solution frameworks include valid inequalities, starting heuristics and primal heuristics. 
In addition, we also developed a genetic algorithm (GA), which is based on a greedy randomized adaptive search procedure (GRASP) version of our starting heuristic. 

In a computational study, we compared our new exact approaches to the previous MIP approached presented in~\cite{MaEtAl2019} and also analyzed the performance of the GRASP and GA. 
The study revealed that our exact solution algorithms are up to 500 times faster compared to the exact approaches of~\cite{MaEtAl2019} and instances with up to 125 vertices can be solved to optimality within a timelimit of 1800 seconds. 
Moreover, the GRASP and GA also works well and often find the optimal or a near-optimal solution within a short runtime. 
In the study, we also investigated the influence of different instance-characteristics, e.g., density and weight-structure on the performance of our approaches. Instances, where the edge weights are in a larger range compared to the vertex weights turned out to be the most difficult for our algorithms, while high density also plays a role in making instances difficult.

The attained results confirm that domination problems are computationally challenging and, therefore, require the combined effort of MIP-based and heuristic approaches in order to tackle more difficult instances.
Therefore, we believe that the development of further modeling and algorithmic advances for domination problem variants is an interesting venue for future work as these problems are relevant both from the methodological and practical point of view.

\paragraph{Acknowledgments} E. \'Alvarez-Miranda acknowledges  the  support  of  the  National Commission for Scientific and Technological Research CONICYT, Chile,  through  the  grant  FONDECYT  N.1180670  and  through  the  Complex  Engineering Systems Institute
PIA/BASAL AFB180003.


\section*{References}

\bibliographystyle{plainnat}
\bibliography{WTDP-bib}

\begin{thebibliography}{36}
\providecommand{\natexlab}[1]{#1}
\providecommand{\url}[1]{\texttt{#1}}
\expandafter\ifx\csname urlstyle\endcsname\relax
  \providecommand{\doi}[1]{doi: #1}\else
  \providecommand{\doi}{doi: \begingroup \urlstyle{rm}\Url}\fi

\bibitem[Billionnet(2005)]{billionnet2005different}
A.~Billionnet.
\newblock Different formulations for solving the heaviest k-subgraph problem.
\newblock \emph{INFOR: Information Systems and Operational Research},
  43\penalty0 (3):\penalty0 171--186, 2005.

\bibitem[Bonomo et~al.(2012)Bonomo, Marenco, Saban, and
  Stier-Moses]{bonomo2012polyhedral}
F.~Bonomo, J.~Marenco, D.~Saban, and N.~Stier-Moses.
\newblock A polyhedral study of the maximum edge subgraph problem.
\newblock \emph{Discrete Applied Mathematics}, 160\penalty0 (18):\penalty0
  2573--2590, 2012.

\bibitem[Cockayne and Hedetniemi(1977)]{doi:10.1002/net.3230070305}
E.~Cockayne and S.~Hedetniemi.
\newblock Towards a theory of domination in graphs.
\newblock \emph{Networks}, 7\penalty0 (3):\penalty0 247--261, 1977.

\bibitem[Cockayne et~al.(1980)Cockayne, Dawes, and
  Hedetniemi]{cockayne1980total}
E.~Cockayne, R.~Dawes, and S.~Hedetniemi.
\newblock Total domination in graphs.
\newblock \emph{Networks}, 10\penalty0 (3):\penalty0 211--219, 1980.

\bibitem[Dey and Molinaro(2018)]{Dey2018}
S.~Dey and M.~Molinaro.
\newblock Theoretical challenges towards cutting-plane selection.
\newblock \emph{Mathematical Programming}, 170\penalty0 (1):\penalty0 237--266,
  2018.

\bibitem[Du and Wan(2013)]{DuWan2013}
D.~Du and P.~Wan.
\newblock \emph{Connected Dominating Set: Theory and Applications}, volume~77
  of \emph{Springer Optimization and Its Applications}.
\newblock Springer, 1st edition, 2013.

\bibitem[Erwin(2004)]{erwin2004dominating}
D.~Erwin.
\newblock Dominating broadcasts in graphs.
\newblock \emph{Bulletin of the Institute of Combinatorics and its
  Applications}, 42:\penalty0 89--105, 2004.

\bibitem[Fischetti et~al.(2016)Fischetti, Ljubi{\'c}, and
  Sinnl]{fischetti2016redesigning}
M.~Fischetti, I.~Ljubi{\'c}, and M.~Sinnl.
\newblock Redesigning benders decomposition for large-scale facility location.
\newblock \emph{Management Science}, 63\penalty0 (7):\penalty0 2146--2162,
  2016.

\bibitem[Giap and Ha(2014)]{giap2014parallel}
C.~Giap and D.~Ha.
\newblock Parallel genetic algorithm for minimum dominating set problem.
\newblock In \emph{2014 International Conference on Computing, Management and
  Telecommunications (ComManTel)}, pages 165--169. IEEE, 2014.

\bibitem[Goddard and Henning(2013)]{GODDARD2013839}
W.~Goddard and M.~Henning.
\newblock Independent domination in graphs: A survey and recent results.
\newblock \emph{Discrete Mathematics}, 313\penalty0 (7):\penalty0 839--854,
  2013.

\bibitem[Hagberg et~al.(2008)Hagberg, Swart, and Chult]{hagberg2008exploring}
A.~Hagberg, P.~Swart, and D.~Chult.
\newblock Exploring network structure, dynamics, and function using networkx.
\newblock Technical report, Los Alamos National Laboratory, Los Alamos, NM,
  United States, 2008.

\bibitem[Haynes et~al.(2013)Haynes, Hedetniemi, and
  Slater]{haynes2013fundamentals}
T.~Haynes, S.~Hedetniemi, and P.~Slater.
\newblock \emph{Fundamentals of domination in graphs}.
\newblock CRC press, 2013.

\bibitem[Hedar and Ismail(2010)]{hedar2010hybrid}
A.~Hedar and R.~Ismail.
\newblock Hybrid genetic algorithm for minimum dominating set problem.
\newblock In D.~Taniar, O.~Gervasi, B.~Murgante, E.~Pardede, and B.~Apduhan,
  editors, \emph{International conference on computational science and its
  applications}, volume 6019 of \emph{Lecture Notes in Computer Science}, pages
  457--467. Springer, 2010.

\bibitem[Henning(2004)]{henning2004restricted}
M.~Henning.
\newblock Restricted total domination in graphs.
\newblock \emph{Discrete Mathematics}, 289\penalty0 (1-3):\penalty0 25--44,
  2004.

\bibitem[Henning(2009)]{henning2009survey}
M.~Henning.
\newblock A survey of selected recent results on total domination in graphs.
\newblock \emph{Discrete Mathematics}, 309\penalty0 (1):\penalty0 32--63, 2009.

\bibitem[Henning and Yeo(2013)]{henning2013total}
M.~Henning and A.~Yeo.
\newblock \emph{Total domination in graphs}.
\newblock Springer, 2013.

\bibitem[Kang(2013)]{Kang2013}
L.~Kang.
\newblock \emph{Variations of Dominating Set Problem}, pages 3363--3394.
\newblock Springer, 2013.

\bibitem[Kramer(2017)]{Kramer2017}
O.~Kramer, editor.
\newblock \emph{Genetic Algorithm Essentials}, volume 679 of \emph{Series on
  Applied Optimization}.
\newblock Springer, 1st edition, 2017.

\bibitem[Laporte et~al.(2015)Laporte, Nickel, and da~Gama]{laporte2015location}
G.~Laporte, S.~Nickel, and F.~S. da~Gama.
\newblock \emph{Location science}, volume 528.
\newblock Springer, 2015.

\bibitem[Laskar et~al.(1984)Laskar, Pfaff, Hedetniemi, and
  Hedetniemi]{laskar1984algorithmic}
R.~Laskar, J.~Pfaff, S.~Hedetniemi, and S.~Hedetniemi.
\newblock On the algorithmic complexity of total domination.
\newblock \emph{SIAM Journal on Algebraic Discrete Methods}, 5\penalty0
  (3):\penalty0 420--425, 1984.

\bibitem[Letchford and S{\o}rensen(2014)]{letchford2014new}
A.~N. Letchford and M.~M. S{\o}rensen.
\newblock A new separation algorithm for the boolean quadric and cut polytopes.
\newblock \emph{Discrete Optimization}, 14:\penalty0 61--71, 2014.

\bibitem[Ma et~al.(2019)Ma, Cai, and Yao]{MaEtAl2019}
Y.~Ma, Q.~Cai, and S.~Yao.
\newblock Integer linear programming models for the weighted total domination
  problem.
\newblock \emph{Applied Mathematics and Computation}, 358:\penalty0 146--150,
  2019.

\bibitem[Macambira and de~Souza(2000)]{macambira2000edge}
E.~Macambira and C.~de~Souza.
\newblock The edge-weighted clique problem: valid inequalities, facets and
  polyhedral computations.
\newblock \emph{European Journal of Operational Research}, 123\penalty0
  (2):\penalty0 346--371, 2000.

\bibitem[Nacher and Akutsu(2016)]{NACHER201657}
J.~Nacher and T.~Akutsu.
\newblock Minimum dominating set-based methods for analyzing biological
  networks.
\newblock \emph{Methods}, 102:\penalty0 57--63, 2016.

\bibitem[Ore(1962)]{Ore1962}
O.~Ore.
\newblock \emph{Theory of Graphs}, volume~38 of \emph{Colloquium Publications}.
\newblock American Mathematical Society, 1st edition, 1962.

\bibitem[Padberg(1989)]{padberg1989boolean}
M.~Padberg.
\newblock The boolean quadric polytope: some characteristics, facets and
  relatives.
\newblock \emph{Mathematical Programming}, 45\penalty0 (1-3):\penalty0
  139--172, 1989.

\bibitem[Pinacho et~al.(2018)Pinacho, Blum, and Lozano]{PINACHODAVIDSON2018860}
P.~Pinacho, C.~Blum, and J.~Lozano.
\newblock The weighted independent domination problem: Integer linear
  programming models and metaheuristic approaches.
\newblock \emph{European Journal of Operational Research}, 265\penalty0
  (3):\penalty0 860--871, 2018.

\bibitem[Rahmaniani et~al.(2017)Rahmaniani, Crainic, Gendreau, and
  Rei]{rahmaniani2017benders}
R.~Rahmaniani, T.~Crainic, M.~Gendreau, and W.~Rei.
\newblock The benders decomposition algorithm: A literature review.
\newblock \emph{European Journal of Operational Research}, 259\penalty0
  (3):\penalty0 801--817, 2017.

\bibitem[Rengaswamy et~al.(2017)Rengaswamy, Datta, and
  Ramalingam]{RengaswamyEtAl2017}
D.~Rengaswamy, S.~Datta, and S.~Ramalingam.
\newblock Multiobjective genetic algorithm for minimum weight minimum connected
  dominating set.
\newblock In A.~Abraham, P.~Muhuri, A.~Muda, and N.~Gandhi, editors,
  \emph{International Conference on Intelligent Systems Design and
  Applications}, volume 736 of \emph{Advances in Intelligent Systems and
  Computing}, pages 558--567. Springer, 2017.

\bibitem[Resende and Ribeiro(2016)]{ResendeRibeiro2016}
M.~Resende and C.~Ribeiro, editors.
\newblock \emph{Optimization by {GRASP}}.
\newblock Springer, 1st edition, 2016.

\bibitem[Rolfes(2014)]{rolfes2014copositive}
J.~Rolfes.
\newblock Copositive formulations of the dominating set problem and
  applications master thesis.
\newblock Master's thesis, Department of Mathematics,Faculty of Science,
  University of Cologne, 2014.

\bibitem[{Sarubbi} et~al.(2016){Sarubbi}, {Mesquita}, {Wanner}, {Santos}, and
  {Silva}]{7502983}
J.~{Sarubbi}, C.~{Mesquita}, E.~{Wanner}, V.~{Santos}, and C.~{Silva}.
\newblock A strategy for clustering students minimizing the number of bus stops
  for solving the school bus routing problem.
\newblock In \emph{2016 IEEE/IFIP Network Operations and Management Symposium},
  pages 1175--1180, 2016.

\bibitem[Sun and Ma(2017)]{SunMa2017}
P.~Sun and X.~Ma.
\newblock Dominating communities for hierarchical control of complex networks.
\newblock \emph{Information Sciences}, 414:\penalty0 247--259, 2017.

\bibitem[Sundar(2014)]{sundar2014steady}
S.~Sundar.
\newblock A steady-state genetic algorithm for the dominating tree problem.
\newblock In G.~Dick, W.~Browne, P.~Whigham, M.~Zhang, L.~Thu, B.~Hisao,
  I.~Yaochu, J.~Xiaodong, L.~Yuhui, S.~Pramod, S. amd~Kay, C.~Tan, and K.~Tang,
  editors, \emph{Asia-Pacific Conference on Simulated Evolution and Learning},
  volume 8886 of \emph{Lecture Notes in Computer Science}, pages 48--57.
  Springer, 2014.

\bibitem[Wan et~al.(2002)Wan, Alzoubi, and Frieder]{wan2002distributed}
P.~Wan, K.~Alzoubi, and O.~Frieder.
\newblock Distributed construction of connected dominating set in wireless ad
  hoc networks.
\newblock In \emph{Proceedings of the 21st Annual Joint Conference of the IEEE
  Computer and Communications Societies}, volume~3, pages 1597--1604, 2002.

\bibitem[Wesselmann and Stuhl(2012)]{WesselmannStuhl2012}
F.~Wesselmann and U.~Stuhl.
\newblock Implementing cutting plane management and selection techniques.
\newblock Technical report, University of Paderborn, Germany, 2012.

\end{thebibliography}

\end{document}